\documentclass[final,11pt,leqno]{amsart}

\usepackage{amsmath,amssymb,amscd,amsthm}
\usepackage{latexsym}
\usepackage{graphicx}

\usepackage{showkeys}

\newtheorem{theorem}{Theorem}
\newtheorem{lemma}{Lemma}
\newtheorem{proposition}{Proposition}

\newtheorem{definition}{Definition}
\newtheorem{corollary}{Corollary}

\newtheorem{claim}{Claim}

\newcommand{\f}[2]{\frac{#1}{#2}}
\newcommand{\dpr}[2]{\langle #1,#2 \rangle}

\newcommand{\al}{\alpha}

\newcommand{\de}{\delta}

\newcommand{\ka}{\kappa}
\newcommand{\la}{\lambda}

\newcommand{\si}{\sigma}

\newcommand{\vp}{\varphi}
\newcommand{\om}{\omega}

\newcommand{\rone}{\mathbf R^1}

\newcommand{\cL}{\mathcal L}

\newcommand{\ct}{\mathbf T}

\newcommand{\cb}{\mathcal B}

\newcommand{\cc}{\mathcal C}
\newcommand{\ch}{\mathcal H}

\newcommand{\p}{\partial}

\newcommand{\beq}{\begin{equation}}
\newcommand{\eeq}{\end{equation}}
\newcommand{\beqna}{\begin{eqnarray*}}
\newcommand{\eeqna}{\end{eqnarray*}}
\newcommand{\beqn}{\begin{equation*}}
\newcommand{\eeqn}{\end{equation*}}
\newcommand{\bp}{\begin{proof}}
\newcommand{\ep}{\end{proof}}
\newcommand{\bprop}{\begin{proposition}}
\newcommand{\eprop}{\end{proposition}}
\newcommand{\bt}{\begin{theorem}}
\newcommand{\et}{\end{theorem}}
\newcommand{\bex}{\begin{Example}}
\newcommand{\eex}{\end{Example}}
\newcommand{\bc}{\begin{corollary}}
\newcommand{\ec}{\end{corollary}}
\newcommand{\bcl}{\begin{claim}}
\newcommand{\ecl}{\end{claim}}
\newcommand{\bl}{\begin{lemma}}
\newcommand{\el}{\end{lemma}}

\begin{document}

\title
[Linear stability for periodic waves of Boussinesq  and KGZ]
{Linear stability analysis for periodic traveling waves of the  Boussinesq equation and the KGZ system }

\author{Sevdzhan Hakkaev}
\author{Milena Stanislavova}
\author{Atanas Stefanov}

\address{Sevdzhan Hakkaev 
Faculty of Mathematics and Informatics, Shumen University, 9712
Shumen, Bulgaria} \email{shakkaev@fmi.shu-bg.net}
\curraddr{Yeditepe University
Fen - Edebiyat Fakultesi, Matematik Bolumu
26 Agustos Yerlesimi Kayisdagi Cad. 
81120 Kayisdagı - Istanbul, Turkey}
\address{Milena Stanislavova
Department of Mathematics, University of Kansas, 1460 Jayhawk
Boulevard,  Lawrence KS 66045--7523} \email{stanis@math.ku.edu}
\address{Atanas Stefanov
Department of Mathematics, University of Kansas, 1460 Jayhawk
Boulevard,  Lawrence KS 66045--7523}

\email{stefanov@math.ku.edu}

\thanks{Hakkaev supported in part by research grant  DDVU 02/91 of 2010 of the
Bulgarian Ministry of Education and Science.
Stanislavova  supported in part by NSF-DMS \# 0807894.
Stefanov supported in part by NSF-DMS \# 0908802.}

\date{\today}



\begin{abstract}
The question for linear stability of spatially periodic waves for the Boussinesq equation (the cases $p=2,3$) and the Klein-Gordon-Zakharov system is considered. For a wide class of solutions, we completely and explicitly characterize their linear stability (instability respectively), when the perturbations are taken with the same period $T$. In particular, our results allow us to completely recover the   linear stability results, in the limit $T\to \infty$,  for the whole line case.
\end{abstract}
\
\maketitle

\section{Introduction}
In this paper we will be interested in the stability of spatially periodic waves for certain models, which involve second derivative in time. Our interest
will be mainly in two PDE - the Boussinesq equation and the Klein-Gordon-Zakharov system, although the methods that we develop here will certainly
find applications in other
 related models.

The Cauchy problem for the Boussinesq equation, with periodic boundary conditions, is
\begin{equation}
\label{a:1}
u_{tt}+u_{xxxx}-u_{xx}+(f(u))_{xx}=0, \ \ (t,x)\in \rone_+\times [0,T],
\end{equation}
where $f(u)$ will for the most part be $f(u)=u^p, p>1$. This is a model  that was derived by Boussinesq, \cite{Bous}, for $p=2$, but was subsequently studied
by many authors, both in the periodic and whole line context. We now review the current results regarding  the well-posedness properties of the Boussinesq equation.
While we have a very satisfactory theory for the local solutions, see below, the  the  global well-posedness does not hold. More precisely, even if one requires smooth
compactly supported data, the solutions may develop singularities in finite time, \cite{Bona}. This makes the stability questions, which is the main subject of this
article even more relevant and interesting.

In the whole line scenario, local well-posedness was established by Bona and Sachs,  in the Sobolev spaces
$H^{\f{5}{2}+}(\rone)\times  H^{\f{3}{2}+}(\rone)$, \cite{Bona}.   Further contributions were made by Tsutsumi and Mathashi, \cite{TM},  Linares, \cite{Linares}
(who also showed global existence for small data). Farah, \cite{Farah} has shown well-posedness in $H^s(\rone)\times \tilde{H}^{s-2}(\rone)$, when $s>-1/4$ and
the space $\tilde{H}^\al$ is defined via  $\tilde{H}^\al=\{u: u_x \in H^{\al-1}(\rone)\}$.   Kishimoto and Tsugava, \cite{KT}  have finally shown well-posedness
for all $s>-1/2$, which is likely to be sharp.

Regarding the case of periodic boundary conditions, Fang and Grillakis, \cite{FG} who have established local well-posedness in $H^s(\ct)\times H^{s-2}(\ct)$, $s>0$ (when
  $1<p<3$ in \eqref{a:1}).  This result was  improved for $p=2$ to $s>-1/4$  by Farah and Scialom, \cite{FS}. Oh and Stefanov have recently shown
  local well-posedness in $H^s(\ct)\times H^{s-2}(\ct)$, $s>-3/8$, \cite{OS}.

 Our other main object of investigation  will be the Klein-Gordon-Zakharov system, which is given  by\footnote{The coefficient $\f{1}{2}$ in front
 of the non-linear term $(|u|^2)_{xx}$ is non-standard, but rather adopted for convenience of our presentation. In particular, it helps create a
 self-adjoint linearized operator, which otherwise can be achieved via a simple change of the time variable.}
  \begin{equation}
  \label{KGZ}
    \left\{ \begin{array}{ll}
      u_{tt}-u_{xx}+u+uv=0, \ \  (t,x)\in \rone\times \rone \ \ \textit{or}\ \ (t,x)\in \rone\times [0,T]\\
       n_{tt}- n_{xx}=\f{1}{2}(|u|^2)_{xx} \\
    \end{array} \right.
   \end{equation}
This system    describes the interaction of a
Langmuir wave and an ion sound wave in  plasma. More precisely,
$u$ is the  fast scale component of the electric field, whereas
$n$ denotes the deviation of ion density, \cite{Zakh, Dendy}. The system \eqref{KGZ} is locally well-posed in various function  spaces,
(see  Ginibre-Tsutsumi-Velo, \cite{GTV} and
Ozawa-Tsutaya-Tsutsumi, \cite{OTT}).  In our previous paper, \cite{HSS}, we have shown that the Cauchy problem \eqref{KGZ} is locally well-posed
(both in priodic and whole line contexts) in $H^\al \times H^{\al-1}\times H^{\al-1}\times H^{\al-2}$, whenever $\al>1/2$.  In  \cite{OTT} and  \cite{OTT1},
the authors have shown that small data produces solutions that persist globally, whereas large solutions are generally expected to blow up in finite time.

The stability of periodic traveling waves have been studied
extensively in the last decade. The nonlinear stability of
periodic waves for Koreteweg-de Vries eqtuation  based on the
Jacobi elliptic functions of cnoidal type   was considered in
\cite{ABS}, and for modified Korteweg-de Vries and nonlinear
Schrodinger equations of dnoidal type in\cite{An2}. In \cite{HIK}, the authors 
  considered the nonlinear stability of periodic waves for
generalized Benjamin-Bona-Mahony equation. Recently,  Arruda
\cite{Ar} considered the nonlinear stability of periodic traveling
waves for Boussinesq equation. The approach is based on the theory
developed in \cite{Be1, Bo} and \cite{GSS} for the stability of
solitary waves.

In this paper, we consider the linearized stability of periodic traveling  waves for the Boussinesq model (the cases $p=2,3$) and the Klein-Gordon-Zakharov system. These are the cases of second order in time models, for which we can explicitly write the solutions in elliptic functions (and moreover, we can explicitly compute the relevant portion of the  spectrum of the   linearized operators). While this certainly helps in the analysis,   we believe that our results should be generalizable to all values of $p>1$. 
The main tool is the theory for linearized stability for such models, developed recently by the second and third author, \cite{SS1}. 

The paper is organized as follows. In Section \ref{sec:11}, we present the construction of our main object of study - the periodic traveling waves. This is not a new material by any means, but we do it in order to single out the solutions of interest\footnote{note that there are solutions, which are not consider herein},  and to introduce some notations. In Section \ref{sec:21}, we setup the linear stability problem, after which we present the main results. In Section \ref{sec:2}, we outline the theory for linearized stability for second order ODE in \cite{SS1},   and point out to the relevant spectral theoretic results about their linearized operators. In Section \ref{sec:133}, we prove the main results - Theorems \ref{theo:1}, \ref{theo:2} about the Boussinesq model, while in Section \ref{sec:5}, we prove Theorem \ref{theo:10} about the KGZ system.

\section{Construction of the periodic traveling waves}
\label{sec:11}
In this section, we show a glimpse of
the construction of the periodic waves  - in Section \ref{sec:1.1}   for the Boussinesq equation (when $p=2,3$) and in Section  \ref{sec:1.2}  for the KGZ system.
\subsection{Construction of the traveling wave solutions for the Boussinesq equation}
\label{sec:1.1}
Applying the traveling wave ansatz,   one sees that there is an one-parameter family of traveling waves of the form $\vp(x-c t), |c|\in (-1,1)$, which obey the equation
$
\p_{xx}[c^2\vp+\vp''-\vp+f(\vp)]=0, \ \ 0\leq x\leq T.
$
whence, there exists $a,C$, so that
$$
c^2\vp+\vp''-\vp+f(\vp)=C x+a,  \ \ 0\leq x\leq T
$$
By the periodicity of $\vp$, we   conclude that $C=0$ and thus, we have a   family of waves   satisfying
\begin{equation}
\label{a:1010}
 \vp''-(1-c^2)\vp+f(\vp)=a, \ \ 0\leq x\leq T.
\end{equation}
We now construct   solutions of \eqref{a:1010} in various cases of interest, most notably $p=2$ and $p=3$. This material is not new, but in order to
introduce the particular parametrization that is convenient for us,   we include a sketch of the construction for completeness.

\subsubsection{The case $p=2$}
We will consider only the symmetric case $a=0$. For the nonlinearity,
 $f(u)=\frac{u^2}{2}$, we have (with $\omega=1-c^2$) 
\begin{equation}
\label{kdv1}
  -w\varphi+{\frac{1}{2}}\varphi^2+\varphi''=0.
\end{equation}
Therefore,
    \begin{equation}\label{kdv1a}
    \varphi'^2-\omega \varphi^2 +\frac13\varphi^3=b
   \end{equation}

Hence the periodic solutions are given by the periodic
trajectories $H(\varphi,\varphi')=b$ of the Hamiltonian vector
field $dH=0$ where
$$H(x,y)=y^2+{\frac{x^3}{3}}-wx^2.$$
The level set $H(x,y)=b$ contains  two periodic trajectories if
$w>0$, $b\in \left( -\frac{2}{3}w^3, 0\right)$ and a unique
periodic trajectory if $b>0$. Here we consider the cases where $b<0$ and
$\varphi_c>0$. To express $\varphi_c$ through elliptic functions,
we denote by $\varphi_1>\varphi_0>0$ the positive solutions of
$\frac13\rho^3-\omega \rho^2-b=0$. Then $\varphi_0\leq
\varphi_c\leq \varphi_0$ and one can rewrite (\ref{kdv1a}) as
\begin{equation}\label{kdv1b}
\varphi_c'^2=\frac13(\varphi_c-\varphi_0)(\varphi_1-\varphi_c)(\varphi_c+\varphi_0+\varphi_1-3\omega).
\end{equation}
Introducing a new variable $s\in(0,1)$ via
$\varphi_c=\varphi_0+(\varphi_1-\varphi_0)s^2$, we transform
(\ref{kdv1b}) into
$$s'^2=\alpha^2(1-s^2)(k'^2+k^2s^2)$$ where $\alpha$, $k$, $k'$
are positive constants ($k^2+k'^2=1$) given by
$$\alpha^2=\frac{2\varphi_1+\varphi_0-3\omega}{12}, \quad
k^2=\frac{\varphi_1-\varphi_0}{2\varphi_1+\varphi_0-3\omega}.$$
Therefore
\begin{equation}\label{kdv2}
\varphi_c(x)=\varphi_0+(\varphi_1-\varphi_0)cn^2(\alpha x;k).
\end{equation}
By the above formulas,
  \begin{equation}\label{kdv2a}
    \begin{array}{ll}
    \varphi_1-\varphi_0=12\alpha^2\kappa^2\\
    \\
    \varphi_1=4\alpha^2(1+\kappa^2)+w\\
    \\
    \varphi_0=4\alpha^2(1-2\kappa^2)+w\\
    \\
    w^2=16\alpha^4(1-\kappa^2+\kappa^4).
   \end{array}
   \end{equation}
The fundamental period of the cnoidal wave $\varphi_c$ in
(\ref{kdv2}) is
 \begin{equation}\label{kdv2b}
   T={\frac{2K(\kappa)}{\alpha}}={\frac{4K(\kappa)\sqrt[4]{1-\kappa^2+\kappa^4}}{\sqrt{w}}}, \; \; \; T\in \left(
   {\frac{2\pi}{\sqrt{w}}}, \infty \right).
 \end{equation}
 Here and below, $K(\kappa)$ and $E(\kappa)$ denote the elliptic
 integrals of the first and second kind in a Legendre form.
 Further, we will use the following relations
   $$
   K'(\kappa)={\frac{E(\kappa)-(1-\kappa^2)K(\kappa)}{\kappa
   (1-\kappa^2)}}, \quad
   E'(\kappa)={\frac{E(\kappa)-K(\kappa)}{\kappa}}.
   $$

\begin{lemma}\label{lkdv1}
For any $w>0$ and $T\in (\f{2\pi}{\sqrt{w}},\infty)$, there is a constant $b=b(w)$ such that
the periodic traveling solution (\ref{kdv2}) has period $T$. The
function $b(w)$ is differentiable.
\end{lemma}
\begin{proof}
It is easily seen that the period $T$ is a strictly
increasing function of $k$:
\begin{eqnarray*}
\frac{d}{dk}(\sqrt[4]{1-\kappa^2+\kappa^4}K(\kappa))&=& \frac{\kappa(2\kappa^2-1)K(\kappa)+2(1-\kappa^2+\kappa^4)K'(\kappa)}{4\sqrt[4]{1-\kappa^2+\kappa^4}^3}\\
\\
&=& \frac{2(1-\kappa^2+\kappa^4)E(\kappa)+(1-\kappa^2)(\kappa^2-2)K(\kappa)}{4\kappa
(1-\kappa^2)\sqrt[4]{1-\kappa^2+\kappa^4}^3}>0.
\end{eqnarray*}
Given $w$ and $b$ in their range, consider the functions
$\varphi_0(w,b), \varphi_1(w,c)$, $k(w,b)$ and $T(w,b)$ given by
the formulas \eqref{kdv2b} and \eqref{kdv2a}.  We obtain
$$\frac{\partial T}{\partial b}=\frac{dT}{dk} \frac{dk}{db}=
\frac{1}{2k}\frac{dT}{dk}\frac{d(k^2)}{db}.$$
Further, using that
${\frac{1}{3}}\varphi_0^3-\omega
\varphi_0^2={\frac{1}{3}}\varphi_1^3-\omega \varphi_1^2$,we have
\begin{eqnarray*}
\frac{d(k^2)}{db}&=& {\frac{3(\varphi_0-w){\frac{\partial \varphi_1}{\partial b}}-3(\varphi_1-w){\frac{\partial \varphi_0}{\partial c}}}{(2\varphi_1+\varphi_0-3w)^2}}\\
\\
&=&{\frac{3w^2(\varphi_0-\varphi_1)}{(\varphi_0^2-2w
\varphi_0)(\varphi_1^2-2w \varphi_1)(2\varphi_1+\varphi_0-3w)^2}}.
\end{eqnarray*}
We see that $\f{\partial T(\omega,b)}{\partial b}\neq 0$,  whence the
implicit function theorem implies the result.
\end{proof}

\subsubsection{The case $p=3$}
We   consider  the "symmetric" case $a=0$ only, with $f(u)=u^3$.
Denote $w=1-c^2\in (0,1)$.
Multiplying by $\vp'$ and integrating implies
       \begin{equation}\label{mkdv1}
         \varphi'^{2}=b+w\varphi^2-{\frac{\varphi^4}{2}}.
       \end{equation}
       Hence the periodic solutions are given by the periodic
trajectories $H(\varphi,\varphi')=b$ of the Hamiltonian vector
field $dH=0$ where
$$
H(x,y)=y^2+{\frac{x^4}{4}}-w\frac{x^2}{2}.
$$
 Then
there are two possibilities:
\begin{itemize}
\item({\it outer case}): for any $b>0$ the
orbit defined by $H(\varphi,\varphi')=b$ is periodic and
oscillates around the eight-shaped loop $H(\varphi,\varphi')=0$
through the saddle at the origin.

\item({\it left and right cases}): for any
$b\in(-\frac12w^2,0)$ there are two periodic orbits defined by
$H(\varphi,\varphi')=b$ (the left and right ones). These are
located inside the eight-shaped loop and oscillate around the
centers at $(\mp\sqrt{w},0)$, respectively.
\end{itemize}
We will consider the left and right cases
of Duffing oscillator only. In these cases,   denote by
$\varphi_1>\varphi_0>0$ the positive roots of the quartic equation
$\frac{z^4}{2}-wz^2-b=0$. Then, up to a translation,
we obtain the respective explicit formulas
\begin{equation}\label{mkdv2}
\varphi(z)=\mp \varphi_1 dn(\alpha z; k),\quad
k^2=\frac{\varphi_1^2-\varphi_0^2}{\varphi_1^2}
=\frac{2\varphi_1^2-2w}{\varphi_1^2}, \quad
\alpha={\frac{\varphi_1}{\sqrt{2}}}, \quad T=\frac{2K(k)}{\alpha}.
\end{equation}
Note that the fundamental period $T$  may be written as follows
\begin{equation}
\label{z:200}
T=\f{2K(\ka)}{\al}=\f{2 K(\ka) \sqrt{2-\ka^2}}{\sqrt{w}},\ \  T\in
\left(\f{\sqrt{2} \pi}{\sqrt{w}}, \infty\right).
\end{equation}
Note that this is a two parameter family of solutions, parametrized explicitly in this case by $\vp_0, \vp_1$, although we shall need a different
parametrization. In fact,
we would like to think of this family as being parametrized (implicitly) in terms of $T$ and $w$, where these two will be independent of each other.
We have the following

\begin{lemma}\label{lmkdv1}  For  $T>\f{\sqrt{2} \pi}{\sqrt{w}}$, there is a constant
$b=b(w)$ such that the periodic traveling-wave solution
$(\ref{mkdv2})$ determined by $H(\varphi,\varphi')=b(w)$ has a
period $T$. In addition, the function $b(w)$ is differentiable.
\end{lemma}
For the proof, see Lemma 3.1 in \cite{HaIlKi}.
\subsection{Construction of the traveling wave solutions for the KGZ system}
\label{sec:1.2}
We are looking for $T$ periodic traveling solutions of the Klein-Gordon-Zakharov system, \eqref{KGZ}.  Thus, we take the ansatz $u(t,x)=\vp_c(x-c t),
n(t,x)=\psi_c(x-c t)$, where we take the speed $c\in (-1,1)$. Plugging  into \eqref{KGZ}, we obtain the following relation between $\psi$ and $\vp$
$$
(c^2-1) \psi''=\f{1}{2}(\vp^2)''\ \ \ 0\leq x\leq T.
$$
Two integrations in $x$  imply
$$
(c^2-1) \psi(x)=\f{1}{2}\vp^2(x)+b x+a,
$$
for some constants $a,b$.
By the periodicity we have that  $b=0$, whence $\psi(x)=-\f{\vp^2(x)+a}{2(1-c^2)}$. For simplicity, we shall consider the case $a=0$ only. That is
\begin{equation}
\label{k:1}
\psi_c=-\f{\vp_c^2}{2(1-c^2)}
\end{equation}
Returning back to the other equation in \eqref{KGZ} and using the relation \eqref{k:1}, we obtain the following equation for $\vp_c$
\begin{equation}
\label{k:5}
-(1-c^2) \vp_c''+\vp_c- \f{\vp_c^3}{2(1-c^2)}=0 \ \ \ 0\leq x\leq T.
\end{equation}
Thus,  as in (1.1.2) after multiplying by $\vp_c$ and integrating we get
  \begin{equation}\label{k:5a}
    \vp_c'^2=b+{\frac{\vp_c^2}{w}}-{\frac{\vp_c^4}{4w^2}}={\frac{1}{4w^2}}(4w^2b+4w\vp_c^2-\vp_c^4).
   \end{equation}
   Denote by $\vp_1>\vp_0>0$ the positive roots of the polynomial $P(z)=z^4-4wz^2-4w^2b.$ Then (\ref{k:5a}) can be written in the form

     \begin{equation}
     \label{k:5b}
       \vp_c'^2={\frac{1}{4w^2}}(\vp_c^2-\vp_0^2)(\vp_1^2-\vp_c^2).
     \end{equation}
     and the solution of \eqref{k:5a} is given by

\begin{equation}
\label{k:10}
\vp_c(x)=\vp_1 dn(\alpha x, \kappa),
\end{equation}
where
  \begin{equation}\label{k:10a}
     \vp_1^2+\vp_0^2=4w, \quad \alpha={\frac{\vp_1}{2w}}, \quad \kappa^2={\frac{\vp_1^2-\vp_0^2}{\vp_1^2}}.
   \end{equation}
   Moreover,
     \begin{equation}\label{k:10c}
       (2-\kappa^2)\vp_1^2=4w, \quad \alpha={\frac{1}{\sqrt{w(2-\kappa^2)}}}, \quad  4w^2b=4w\vp_1^2-\vp_1^4.
     \end{equation}
   Since $dn$ has fundamental period $2K(\kappa)$, then the solution $\vp_c$ has fundamental period $T=\frac{2K(\kappa)}{\alpha}$. In terms of $\ka, w$,  given by
     \begin{equation}\label{k:10b}
       T= 2K(\ka) \sqrt{2-\ka^2} \sqrt{w} , \quad T\in I= ( \sqrt{2}\pi\sqrt{w}  , \infty)
     \end{equation}
\begin{lemma}\label{lk}
   For any $w>0$ and $T>\sqrt{2}\pi\sqrt{w}  $, there is a constant $b=b(w)$ such that the periodic traveling solution (\ref{k:10}) has period $T$.
 \end{lemma}
 \begin{proof}  The period $T$ is a strictly increasing function of $\kappa $:
   $$
   {\frac{d}{d\kappa}}[\sqrt{2-\kappa^2}K(\kappa)]={\frac{K'(\kappa)+E'(\kappa)}{\sqrt{2-\kappa^2}}}>0.
   $$
   From (\ref{k:10a}) and (\ref{k:10c}), we have

     $${\frac{dT}{db}}={\frac{dT}{d\kappa}}{\frac{d\kappa}{db}}={\frac{1}{2\kappa}}{\frac{dT}{d\kappa}}{\frac{d\kappa^2}{db}}$$

     $${\frac{d\kappa^2}{db}}={\frac{d\kappa^2}{d\vp_1^2}}{\frac{d\vp_1^2}{db}}={\frac{16w^2}{\vp_1^4(4w-2\vp_1^2)}}.$$
     The implicit function theorem then implies  the result.

   \end{proof}

\section{Results}
\label{sec:21}
\subsection{Setting the linear stability problem for the Boussinesq equation}
We now set up the linear stability/instability  problem for \eqref{a:1}.  Set the ansatz $u=\vp_c(x+c t)+v(t, x+ct)$ and ignore all terms $O(v^2)$.
We get   $v_{tt}+2 c v_{t x}+ M v=0$, where
$$
M v=\p_x^4 v -(1-c^2)\p_x^2 v + (f'(\vp_c) v)_{xx}
$$
Note that this operator $M$ is not self-adjoint. However, if we introduce the variable $z: z_x=v$, we get the following linearized equation in terms of $z$,
\begin{equation}
\label{zz:2}
z_{ttx}+2 c z_{t xx}+ M[z_x]=0.
\end{equation}
Note that $M[z_x]=\p_x[H[z]]$, where
\begin{equation}
\label{z:2}
H_c z=\p_x^4 z-(1-c^2)\p_x^2 z + (f'(\vp_c) z_x)_{x}
\end{equation}
Thus, the linearized equation becomes $\p_x[z_{tt}+2 c z_{tx}+ H z]=0$. In our considerations, we say that the wave $\vp_c$ is spectrally unstable,
exactly when there is an exponentially growing mode, that is a pair $\la\in \cc:\Re\la>0$ and  a  $T$ periodic function $\psi\in H_{per}^4(0,T)$,
so that $\p_x[\la^2\psi+2 c \la \psi'+ H \psi]=0$. This of course implies upon integration that for some constant $a$
$$
\la^2\psi+2 c \la \psi'+ H \psi=a.
$$
Integrating in $[0,T]$ and taking into account that both $\psi', H\psi$ are exact derivatives, impies that $a=\f{\la^2 \int_0^T \psi(x)dx}{T}$.
Thus letting $\tilde{\psi}:=\psi-\f{a}{\la^2}$ implies that $\int_0^T\tilde{\psi}(x) dx=0$ and
$$
\la^2\tilde{\psi}+2 c \la \tilde{\psi}+ H \tilde{\psi}=0
$$
These arguments motivate the following
\begin{definition}
\label{defi:1}
We say that the traveling wave $\vp_c$ is spectrally/linearly unstable, if there exists
an $T$ periodic, {\bf mean value zero} function $\psi\in D(H_c)$ and $\la:\Re\la >0$,
so that
\begin{equation}
\label{55}
\la^2\psi+2 c \la\psi'+ H_c\psi=0.
\end{equation}
\end{definition}
 The question for linear stability of equations in the form
 \begin{equation}
 \label{a:10}
 z_{tt}+2\mu z_{tx}+\ch z=0,
 \end{equation}
 or what is equivalent (at least in this case) to the solvability  of
 \begin{equation}
\label{5}
\la^2\psi+2 \mu \la\psi'+ \ch\psi=0.
\end{equation}
in $L^2_0(0,T)=\{f\in L^2_{per}(0,T): \int_0^T f(x) dx=0\}$
 has been addressed in a recent paper by the second and third author, \cite{SS1}.
 Note that  the self-adjoint  operator  $\ch$ that appears in \eqref{z:2}  is in the form
 $$
 \ch=-\p_x \cL \p_x, \ \ \cL=-\p_x^2+(1-c^2)-f'(\vp_c).
 $$
 Here $\cL$ is the ubiquitous  standard second order Schr\"odinger operator, which appears in the linearization of the generalized KdV equation around
 its traveling wave solution $\vp_c$. This observation will be crucial for the spectral properties of the operator $\ch$ as the properties of $\cL$ are
 generally well-known, at least for the cases into consideration, $p=2,3$.
 \subsection{Setting the linear stability problem for the KGZ system}
 We linearize the KGZ system as follows. We take $u(t,x)=\vp_c(x-c t)+v(t, x-c t)$ and $n(t,x)=\psi_c(x-ct)+h(t,x-ct)$ and ignore the contributions of
 all quadratic and higher order terms. We obtain the following {\it linear system} for the corrections $v,h$,
 \begin{equation}
 \label{k:15}
 \left\{\begin{array}{l}
 v_{tt}-2c v_{t x}-(1-c^2) v_{xx}+v+\psi_c v+\vp_c h=0 \\
 h_{tt}-2c h_{t x}-(1-c^2) h_{xx}-(\vp_c v)_{xx}=0.
 \end{array}
 \right.
 \end{equation}
 Further, introducing a new mean-value zero function $z$, so that $h=z_x$ and
 $w=1-c^2$. The second equation in \eqref{k:15} becomes
 $$
 \p_x[z_{tt}-2c z_{t x}-(1-c^2) z_{xx}-(\vp_c z_x)_{x}]=0,
 $$
 whence integrating in $x$ yields
 $$
 z_{tt}-2c z_{t x}-(1-c^2) z_{xx}-(\vp_c z_x)_{x}=a(t)
 $$
  for some function $a(t)$. Observe however, that in our choice of $z$, we have required that \\ $\int_0^T z(t,x) dx=0$. Thus, integrating the last
  equation in $[0,T]$ yields that all integrals on the left are zero\footnote{each term is either an exact derivative or $z_{tt}$, which is mean value zero}
   $a(t)=0$, whence $z_{tt}-2c z_{t x}-(1-c^2) z_{xx}-(\vp_c z_x)_{x}=0$. We have shown that one can rewrite the linear stability problem \eqref{k:15} as follows
 \begin{equation}
\label{k:20}
\vec{\Phi}_{tt}-2c \vec{\Phi}_{tx}+\ch \vec{\Phi}=0,
\end{equation}
where $\vec{\Phi}=\left(\begin{array}{c} v \\ z \end{array}\right)$ and
\begin{equation}
\label{k:22}
\ch=\left(\begin{array}{cc} H_1 & A \\ A^* & H_2 \end{array}\right),
\end{equation}
where
\begin{eqnarray*}
& & H_1=-(1-c^2)\p^2_x+1 +\psi_c= -w \p^2_x +1 -\f{\vp^2_c}{2 w};  \\
& & H_2=-w\p^2_x; \\
& & A z =\vp_c z_x;  \ \ A^* z=-(\vp_c z)_x.
\end{eqnarray*}
Clearly, the operator $\ch$ is self-adjoint, when considered over the domain
$$
D(\ch)=H^2[0,T]\times H^2_0[0,T] \subset L^2[0,T]\times L^2_0[0,T]
$$
so that $\ch: D(\ch)\to  L^2[0,T]\times L^2_0[0,T]$. Note that , when considering this spectral problem, our basic Hilbertian  space will be
$L^2[0,T]\times L^2_0[0,T]$, instead of the usual $L^2[0,T]\times L^2[0,T]$.

 \subsection{Precise formulation of the main results}
 Our main results are described   in the next Theorems. Note that in all of them, the issue is linear stability for the said families of spatially periodic
 solutions, {\it when the perturbation is taken to be periodic with the same period as the underlying solution}.

 The issue for stability/instability, when perturbations are taken to be with period of the form $nT$, $n$ integer is a more complicated one.
 Clearly, the instability results continue to apply in this case, but it may very well be that some, previously stable waves (in the context of the same period perturbations) become unstable, when perturbed by $nT$ periodic functions.

 \begin{theorem}
 \label{theo:1}
 Let the nonlinearity in \eqref{a:1} has the form $f(u)=u^3$.
 Then, the two parameter family of dnoidal solutions,  described in \eqref{mkdv2}, are linearly stable, \underline{if and only if}
 $$
 |c|\geq \sqrt{\f{M(\ka)}{4+M(\ka)}}, \ka\in (0,1)
 $$
 where
 $$
M(\ka):=\f{\left[4E(\ka)-\f{\pi^2}{K(\ka)}\right][(2-\ka^2) E(\ka)-2(1-\ka^2)K(\ka)]}{(2-\ka^2)(E^2(\ka)-(1-\ka^2) K(\ka))}, \ka\in (0,1),
$$
where $E(\ka), K(\ka)$ are  elliptic
 integrals of the first and second kind in a Legendre form.
 \end{theorem}
   \begin{figure}[h1]
\centering
\includegraphics[width=8cm,height=6cm]{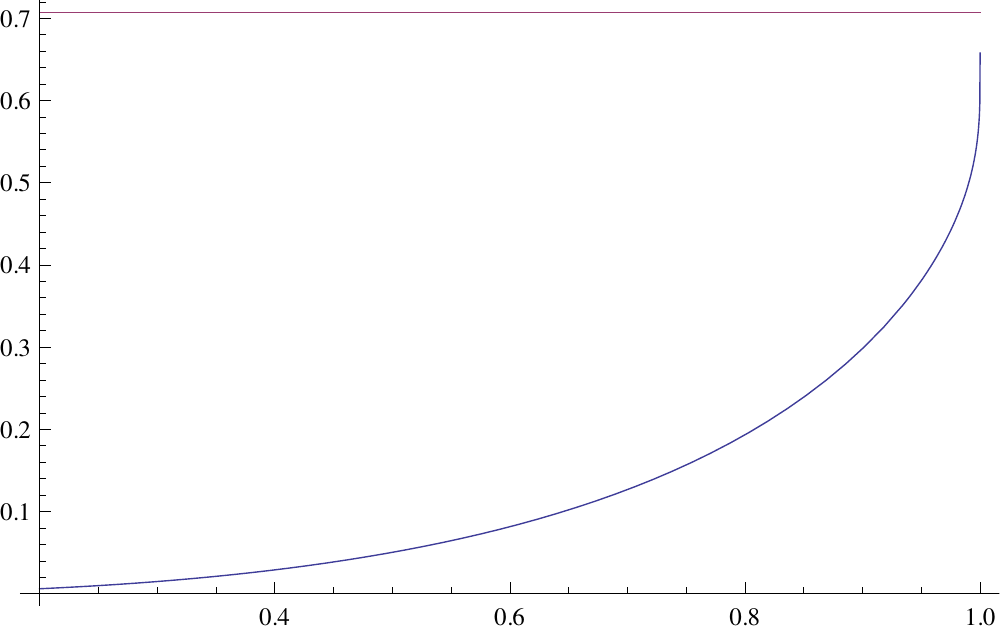}
\caption{Graph of the positive function $\sqrt{\f{M(\ka)}{4+M(\ka)}}$, together with its terminal value  $\f{\sqrt{2}}{2}$ as a reference}
\label{fig1}
\end{figure}
We now give a different formulation of the main result. Let $T>\sqrt{2} \pi$. Then, the waves described in \eqref{mkdv2} are one parameter family of waves, having a fundamental period $T$, which can be
parametrized by $c: |c|<\sqrt{1-\f{2\pi^2}{T^2}}$,   (note that $c, \ka$ are in one-to-one relation given by \eqref{z:200}).   Now, Theorem \ref{theo:1} asserts that the stable waves in this family are exactly those with $|c|\geq c_T$, where
$c_T\in (0,\sqrt{1-\f{2\pi^2}{T^2}})$ is determined as follows. Let $\ka_T$ be the unique solution of
$$
K(\ka) \sqrt{2-\ka^2} \sqrt{4+M(\ka)}=T
$$
Then, $c_T=\sqrt{\f{M(\ka_T)}{4+M(\ka_T)}}$.
\begin{figure}[h2]
\centering
\includegraphics[width=8cm,height=6cm]{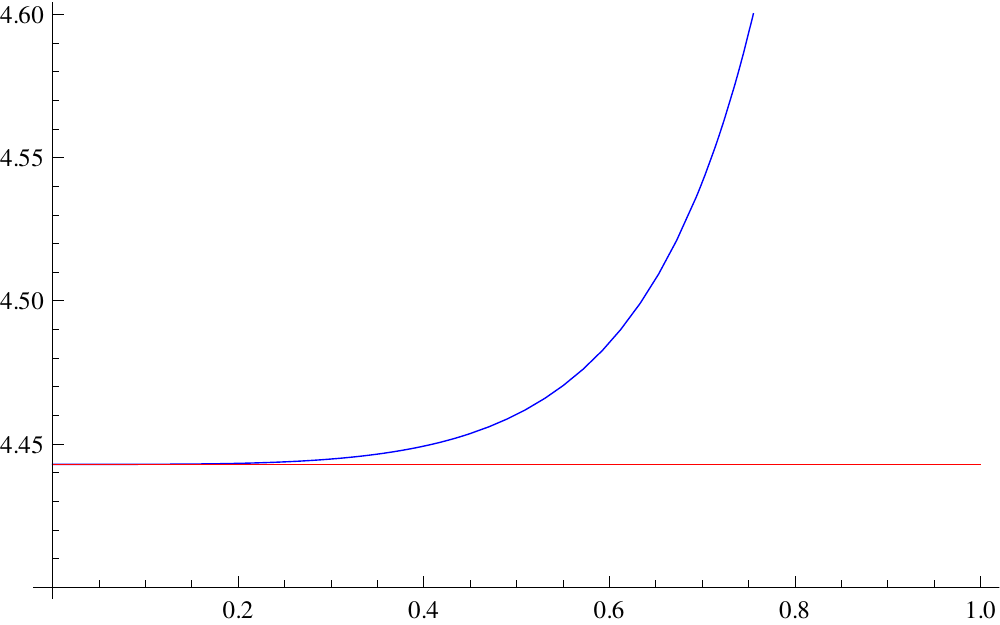}
\caption{The blue line is the graph of the increasing function $ K(\ka) \sqrt{2-\ka^2} \sqrt{4+M(\ka)}:(0,1)\to \rone$. The red line is $\sqrt{2}\pi$. Note that the range of this function is $(\sqrt{2}\pi, \infty)$.}
\label{fig2}
\end{figure}

 Our next theorem concerns the quadratic case
 \begin{theorem}
\label{theo:2}
 Let the nonlinearity in \eqref{a:1} has the
 form $f(u)=\frac{u^2}{2}$. Then the periodic solutions \eqref{kdv2}, are linearly stable, \underline{if and only if}
 $$
 |c|\geq \sqrt{\f{\widetilde{F}(\ka)}{4+\widetilde{F}(\ka)}}, \ka\in (0,1)
 $$
 where
 \begin{eqnarray*}
 & & \tilde{F}(\ka)=\f{\left[
       2F(\kappa)-{\frac{F^2(\kappa)}{16\sqrt{1-\kappa^2+\kappa^4}K^2(\kappa)}}\right] }{ F(\kappa)+256K^4(\kappa)F'(\kappa)G(\kappa)(1-\kappa^2+\kappa^4)+\f{4096 K^6(\ka)(1-\ka^2+\ka^4)^{3/2}(F'(\ka) G(\ka))^2}{1-16\sqrt{1-\kappa^2+\kappa^4}K^2(\kappa)F'(\kappa)G(\kappa)}}
 \end{eqnarray*}
 and
 \begin{eqnarray*}
 & & F(\kappa)=16K(\kappa)[3E(\kappa)+(\kappa^2-2+\sqrt{1-\kappa^2+\kappa^4})K(\kappa)], \\
 & & G(\ka)=\frac{1}{128{\frac{d[K^4(\kappa)(1-\kappa^2+\kappa^4)]}{d\kappa}}}.
 \end{eqnarray*}
\end{theorem}
\begin{figure}[h3]
\centering
\includegraphics[width=8cm,height=6cm]{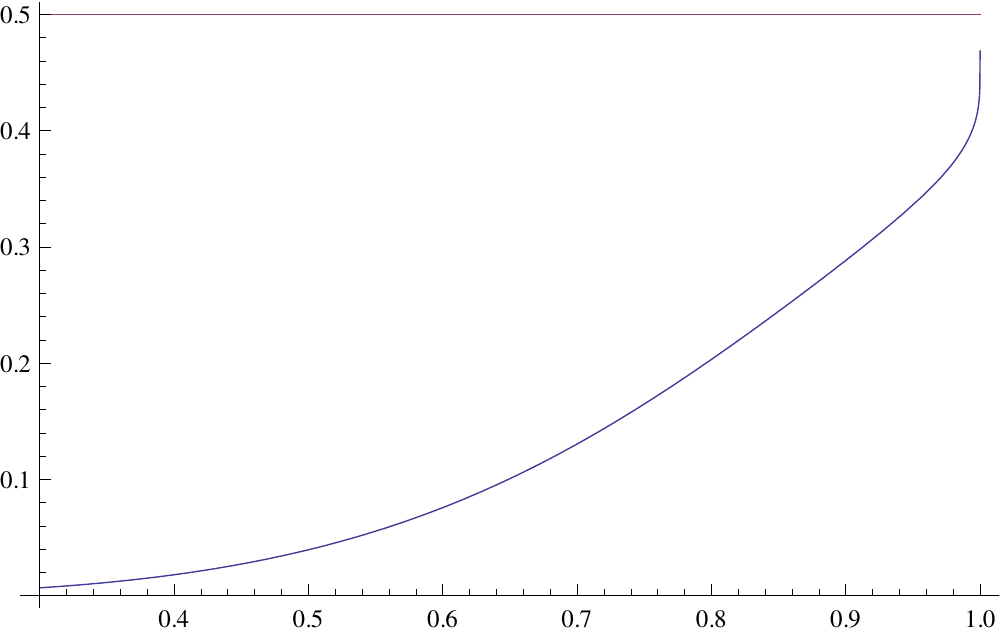}
\caption{Graph of the positive function $\sqrt{\f{\tilde{F}(\ka)}{4+\tilde{F}(\ka)}}$, together with its terminal value $\f{1}{2}$}
\label{fig3}
\end{figure}
An alternative formulation is the following. Fix  $T>2\pi$ and consider the   family of cnoidal solutions, described in  \eqref{kdv2}. This is a one-parameter family of solutions\footnote{where $\ka$ and $c$ are related by \eqref{kdv2b}}, having fundamental period $T$ and indexed by say $c$, where $c: c<\sqrt{1-\f{4\pi^2}{T^2}}$. Then, the linearly stable waves with period $T$ are exactly  those, for which
 $$
 \sqrt{1-\f{4\pi^2}{T^2}}>|c|\geq c_T,
 $$
 where $c_T\in (0, \sqrt{1-\f{4\pi^2}{T^2}})$ is determined as follows. Take $\ka_T$ to  be the unique solution of the algebraic equation
 $$
 2K(\ka)\sqrt[4]{1-\ka^2+\ka^4}\sqrt{4+\tilde{F}(\ka)}=T.
 $$
 Then $c_T= \sqrt{\f{\widetilde{F}(\ka_T)}{4+\widetilde{F}(\ka_T)}}$.
 \begin{figure}[h4]
\centering
\includegraphics[width=8cm,height=6cm]{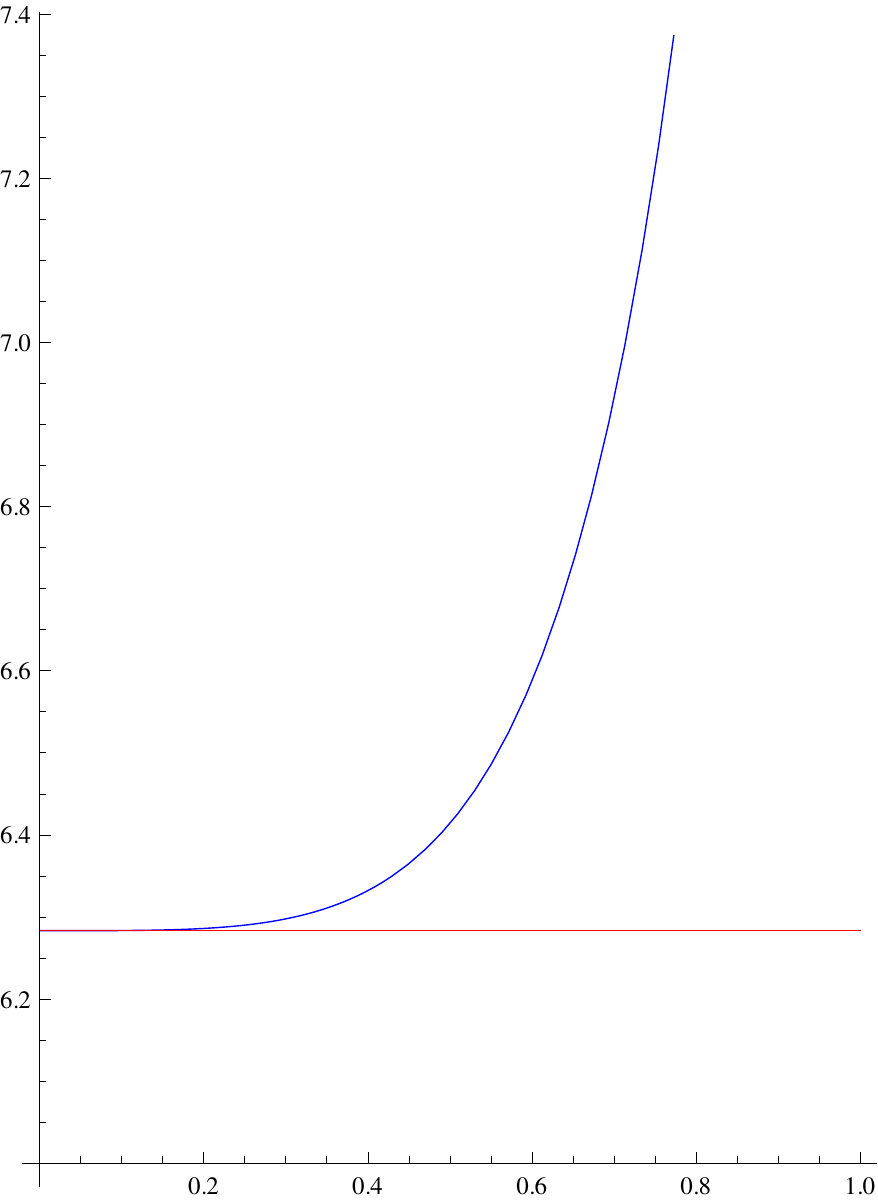}
\caption{The blue line is the graph of the function $2K(\ka)\sqrt[4]{1-\ka^2+\ka^4}\sqrt{4+\tilde{F}(\ka)}:(0,1)\to \rone$. The red line is $2\pi$. Note that the range of this function is $(2\pi, \infty)$.}
\label{fig4}
\end{figure}

 {\bf Remark:} Using the results of Theorem \ref{theo:1} and Theorem \ref{theo:2},
 one can reconstruct the results on linear stability of the whole line
 waves, \cite{SS1}
 \begin{equation}
 \label{k:55}
\vp_c(\xi)=\left[\left(\f{p+1}{2}\right) (1-c^2)  \right]^{\f{1}{p-1}}
 sech^{\f{2}{p-1}}\left(\f{\sqrt{1-c^2} (p-1)}{2} \xi\right).
\end{equation}
Recall that the results of \cite{SS1} predict linear stability if and only if $|c|\geq \f{\sqrt{p-1}}{{2}}$.

Take $p=3$. Then, the periodic waves described in
\eqref{mkdv2} in the limit $\ka\to 1-$ correspond to the whole line waves described in  \eqref{k:55}. Observe that since $\lim_{\ka\to 1-} E(\ka)=1,
\lim_{\ka\to 1-} K(\ka)=\infty$ and $\lim_{\ka\to 1-} (1-k^2)K(\ka)=0$, we conclude easily $\lim_{\ka\to 1-} M(\ka)=4$, whence
$$
\lim_{\ka\to 1-} \sqrt{\f{M(\ka)}{4+M(\ka)}}=\f{\sqrt{2}}{2}.
$$
Similarly, one can check that for $p=2$ (more precisely if $f(u)=\f{u^2}{2}$)),
 we have
 $$
\lim_{\ka\to 1-} \sqrt{\f{\tilde{F}(\ka)}{4+\tilde{F}(\ka)}}=\f{1}{2}.
$$
Thus, we obtain the results in \cite{SS1} for $p=2,3$ as a corollary.

Our next result concerns the KGZ system \eqref{KGZ}.
\begin{theorem}
\label{theo:10}
The KGZ system \eqref{KGZ} has a two-parameter family of traveling wave solutions \\  $(\vp_c, \psi_c)$, described in \eqref{k:1} and \eqref{k:10}.
  These waves are stable, \underline{if and only if} $\ka\in (\ka_0,1)$ and 
$$
1>|c|\geq \f{1}{\sqrt{1+4 N(\ka)}},  
$$
where $\ka_0=0.937095...$ and the function $N$ is defined in \eqref{z001}.
\begin{figure}[h5]
\centering
\includegraphics[width=8cm,height=6cm]{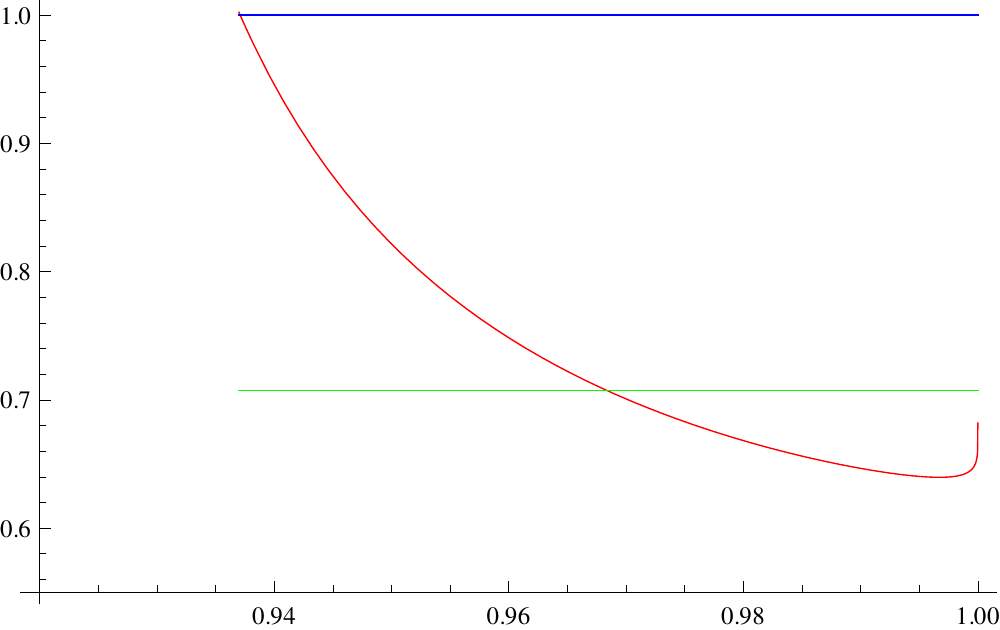}
\caption{The red line is the graph of the function $\f{1}{\sqrt{1+4 N(\ka)}}$ in $(\ka_0,1)$. The green line is the terminal value of $\f{\sqrt{2}}{2}$.}
\label{fig5}
\end{figure}
\end{theorem}
If we take a limit as $\ka\to 1$, we have
$$
\lim_{\ka\to 1} \f{1}{\sqrt{1+4 N(\ka)}}=\f{\sqrt{2}}{2},
$$
Since $\ka\to 1$ corresponds to the case $T=\infty$ or the case of the whole line, this allows us to conclude that the corresponding whole line solitons are stable, provided $|c|\geq \f{\sqrt{2}}{2}$. This was indeed the conclusion in \cite{SS1}, so we are able to deduce this result, as a consequence of Theorem \ref{theo:10}.

Once again, we provide an alternative interpretation of Theorem \ref{theo:10}. Let $T>0$ be a fixed period. Then, there exists a one parameter family of periodic waves with period $T$, described
 in \eqref{k:1}, \eqref{k:10}.  More precisely, this family may be  parameterized\footnote{The parameters $\ka$ and $c$ are related by \eqref{k:10b}}  by $c$ with the following restrictions on $c$: if $T<\sqrt{2} \pi$, then
 $1>|c|>\sqrt{1-\f{T^2}{2\pi^2}}$, otherwise if $T\geq \sqrt{2} \pi$, $c\in (-1,1)$. Then, the stable waves in this family are given by
 $$
 |c|\geq c_T= \f{1}{\sqrt{1+4 N(\ka_T)}},
 $$
 where $\ka_T\in (\ka_0,1)$ is found as the unique solution (see graph below) of the algebraic equation
 $$
 \f{4 K(\ka) \sqrt{2-\ka^2} \sqrt{N(\ka)}}{\sqrt{1+4 N(\ka)}}=T.
 $$
 \begin{figure}[h6]
\centering
\includegraphics[width=8cm,height=6cm]{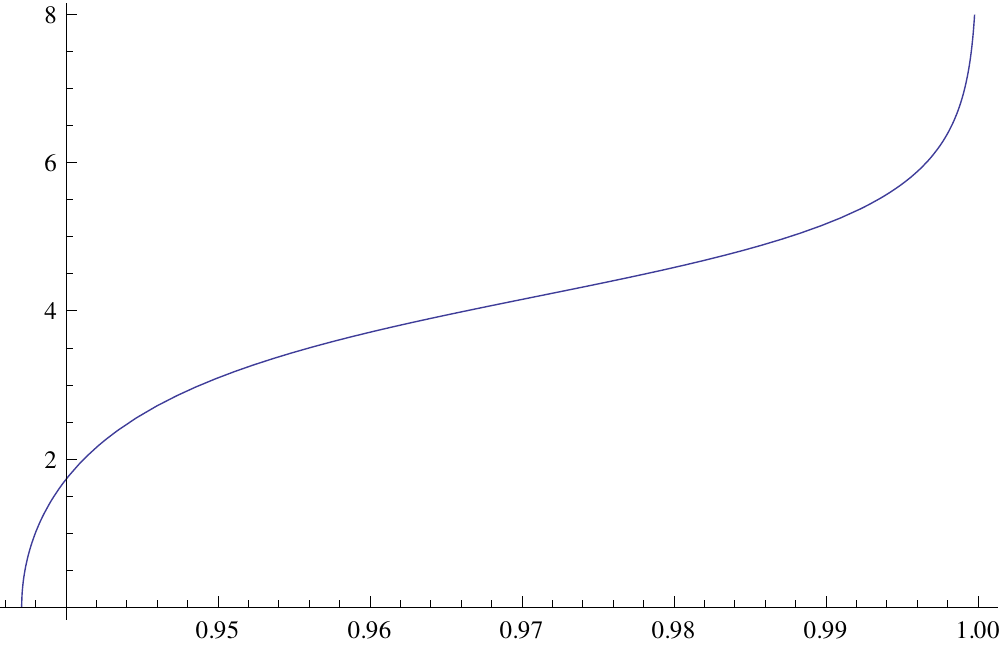}
\caption{Graph of the increasing function
$\f{4 K(\ka) \sqrt{2-\ka^2} \sqrt{N(\ka)}}{\sqrt{1+4 N(\ka)}}:(\ka_0, 1)\to \rone_+$. Note that its range is $(0,\infty)$.}
\label{fig6}
\end{figure}

\section{Preliminaries}
\label{sec:2}

\subsection{Linear stability theory   for second order equations} In this section, we give a precise statements of the results of \cite{SS1},
concerning the linear stability of \eqref{a:10} or what is equivalent the solvability of \eqref{5}.
 We assume the following about the self-adjoint operator $\ch$:
 \begin{equation}
\label{A}
\left\{ \begin{array}{l}
\si(\ch)=\{-\de^2\} \cup \{0\}\cup \si_+(\ch), \si_+(\ch)\subset (\si^2, \infty), \si>0\\
 \ch \phi=-\de^2 \phi,  \ \dim[Ker(\ch+\de^2)]=1 \\
  \ch \psi_0=0, \ \  \dim[Ker(\ch)]=1\\
   \|\psi_0\|=1
\end{array}
\right.
\end{equation}
 Next, we require that for all $\tau>>1$ (note $H+\tau>0$ will be invertible)
\begin{equation}
 \label{E}
   (\ch+\tau)^{-1/2} \p_x (\ch+\tau)^{-1/2}, (\ch+\tau)^{-1} \p_x \in \cb(L^2),
 \end{equation}
Finally, we require
\begin{equation}
\label{B}
  \overline{\ch h}=\ch \bar{h}.
\end{equation}
Note that the last identity ensures that $H$ maps real-valued functions into real-valued functions.
The following theorem, in a more general form,  appears as Theorem 1 in \cite{SS1}.
\begin{theorem}
\label{theo:5}
 Let $\ch$  be a self-adjoint operator on
 a Hilber space $H$. Assume that  it satisfies the structural assumptions \eqref{A}, \eqref{E} as well as the
  the reality assumption \eqref{B}.

Then, if $\dpr{\ch^{-1}[\psi_0']}{\psi_0'}\geq  0$, one has a solution to \eqref{5} for all values  of $\mu\in \rone$, that is instability
in the sense of Definition \ref{defi:1}.
 Otherwise, supposing $\dpr{\ch^{-1}[\psi_0']}{\psi_0'}<0$,
 \begin{itemize}
 \item
  the problem \eqref{5} has solutions, if
   $\mu$ satisfies the inequality
 \begin{equation}
 \label{m:2}
 0\leq |\mu|<
 \frac{1}{2\sqrt{\dpr{-\ch^{-1}[\psi_0']}{\psi_0'}}}=:\mu^*(\ch)
 \end{equation}
 \item  the problem \eqref{5} does not have solutions (i.e. stability), if
   $\om$ satisfies the reverse inequality
   \begin{equation}
 \label{m:3}
  |\mu|\geq  \mu^*(\ch)
 \end{equation}
 \end{itemize}
 \end{theorem}
 {\bf Remarks:}
 \begin{enumerate}
 \item Theorem \ref{theo:5} appears in \cite{SS1} as a result about the stability of
 \eqref{a:10}, but we state it in its equivalent form for solvability of
 \eqref{5}.
 \item In the applications that we consider, we restrict our attention to the Hilbert space $H=L^2_0[0,T]$ or $H=L^2[0,T]\times L^2_0[0,T]$,
 depending on the situation that we are in.
 \end{enumerate}

\subsection{Spectral theory for the Schr\"odinger operators   of   Boussinesq waves}
We   review and state the main results regarding the spectral theory for  the second order Schr\"odinger operators, arising in the
linearization around Boussinesq waves. We consider again the cases $p=2$ and $p=3$ separately.

\subsubsection{The case $p=2$}
In this section, we present some spectral results for $\cL$ that will be useful in the sequel.  The first is a technical lemma,
that will be used to establish the simplicity of the zero eigenvalue for $\ch=-\p_x \cL \p_x$.
\begin{lemma}\label{l1}
$ \dpr{\cL^{-1}1}{1} \neq 0.$
\end{lemma}
\begin{proof}
In this case $\mathcal{L}=-\partial_x^2+w-\varphi_c$ and $Ker
\mathcal{L}=span \varphi_c'$. The spectral properties of the
operator  $\Lambda=-\frac{d^2}{dy^2}-4(1+k^2)+12k^2sn^2(y;k)$ in
$[0,2K(k)]$ are well-known \cite{HIK}. The first three(simple)
eigenvalues and corresponding eigenfunctions of $\Lambda$ are
  $$\begin{array}{ll}
    \mu_0=\kappa^2-2-2\sqrt{1-\kappa^2+4\kappa^4}<0, \\
    \psi_0(x)=dn( x;\kappa)[1-(1+2\kappa^2-\sqrt{1-\kappa^2+4\kappa^4})sn^2( x;\kappa)]>0\\
    \\
    \mu_1=0\\
    \psi_1(x)=dn( x;\kappa)sn(y;\kappa)cn(\alpha x;\kappa)={\frac{1}{2}}{\frac{d}{dy}}cn^2(y;\kappa)\\
    \\
    \mu_2=\kappa^2-2+2\sqrt{1-\kappa^2+4\kappa^4}>0\\
    \psi_2(x)=dn( x;\kappa)[1-(1+2\kappa^2+\sqrt{1-\kappa^2+4\kappa^4})sn^2(y;\kappa)].
   \end{array}
  $$
  Since the eigenvalues of $\mathcal{L}$ and $\Lambda$ are related by
  $\lambda_n=\alpha^2 \mu_n$, it follows that the first three eigenvalues of the operator
      $\mathcal{L}$, equipped with periodic boundary condition on $[0,2K(k)]$
      are simple and $\lambda_0<0, \lambda_1=0, \lambda_2>0$. The
      corresponding eigenfunctions are $\psi_0(\alpha x),
      \psi_1(\alpha x)=const. \varphi'$ and $\psi_2(\alpha x)$.
Note that $1, \varphi_c \perp Ker \mathcal{L}$ and
  \begin{equation}\label{kdv3}
    \mathcal{L}(1)=w-\varphi_c
  \end{equation}
and hence, we can take inverse in (\ref{kdv3})
  \begin{equation}\label{kdv4}
    1=w\mathcal{L}^{-1}1-\mathcal{L}^{-1}\varphi_c.
  \end{equation}
  Taking dot product with $1$ yields the relation
    $$\langle \mathcal{L}^{-1}1,1\rangle = {\frac{1}{w}}\langle
    1,1\rangle + {\frac{1}{w}}\langle \mathcal{L}^{-1}\varphi_c,
    1\rangle$$
    Differentiating (\ref{kdv1}) with respect to $c$, we get $\mathcal{L}[\frac{d\varphi_c}{dc}]=2c\varphi_c$, whence
    \begin{equation}\label{kdv5}
         \mathcal{L}^{-1}\varphi_c={\frac{1}{2c}}{\frac{d\varphi_c}{dc}}.
     \end{equation}
 Entering this last formula in the expression
  for $\langle \mathcal{L}^{-1}1,1\rangle$,
   \begin{equation}\label{kdv7}
   \langle \mathcal{L}^{-1}1,1\rangle
   ={\frac{1}{w}}T+{\frac{1}{2cw}}\left(\p_c \int_{0}^{T}{\varphi_c}dx\right)
   \end{equation}
   Using (\ref{kdv2a}) and that $\int_{0}^{2K(\kappa)}{cn^2(y;
   \kappa)}dy={\frac{2}{\kappa^2}}[E(\kappa)-(1-\kappa^2)K(\kappa)]$,
   we get
     \begin{equation}\label{kdv7b}
     \int_{0}^{T}{\varphi_c}dx=8\alpha[3E(\kappa)+(\kappa^2-2+\sqrt{1-\kappa^2+\kappa^4})K(\kappa)]={\frac{1}{T}}F(\kappa),
     \end{equation}
     where
     $$
     F(\kappa)=16K(\kappa)[3E(\kappa)+(\kappa^2-2+\sqrt{1-\kappa^2+\kappa^4})K(\kappa)].
     $$
   Now, we need to compute
   $$
   \p_c F(\ka)=F'(\ka) \frac{d\kappa}{dw}
   \frac{d w}{d c}=-2c F'(\ka) \frac{d\kappa}{dw}.
   $$
    Thus, to compute $\frac{d\kappa}{dw}$, we differentiate with respect to $w$ the following relation
    \begin{equation}
    \label{n:10}
    w^2=16 \alpha^4(1-\kappa^2+\kappa^4)=
     256 \f{K^4(\ka)(1-\kappa^2+\kappa^4)}{T^4},
     \end{equation}
  obtained from \eqref{kdv2a} and \eqref{kdv2b}.  We get
       \begin{equation}\label{kdv7a}
       {\frac{d\kappa}{dw}}=wT^4G(\kappa),
       \end{equation}
       where
       $$G(\kappa)={\frac{1}{128{\frac{d[K^4(\kappa)(1-\kappa^2+\kappa^4)]}{d\kappa}}}}.$$
    From the above relations, \eqref{kdv2b} and \eqref{n:10},  and  we have
      \begin{equation}
      \label{kdv8}
         \langle
         \mathcal{L}^{-1}1,1\rangle={\frac{T}{w}}-T^3F'(\kappa)G(\kappa)={\frac{T}{w}}[1-16\sqrt{1-\kappa^2+\kappa^4}K^2(\kappa)F'(\kappa)G(\kappa)].
       \end{equation}
       The expression in the   brackets  above is strictly positive (see the graphic below)
   which proves the lemma.
   \begin{figure}[h7]
\centering
\includegraphics[width=8cm,height=6cm]{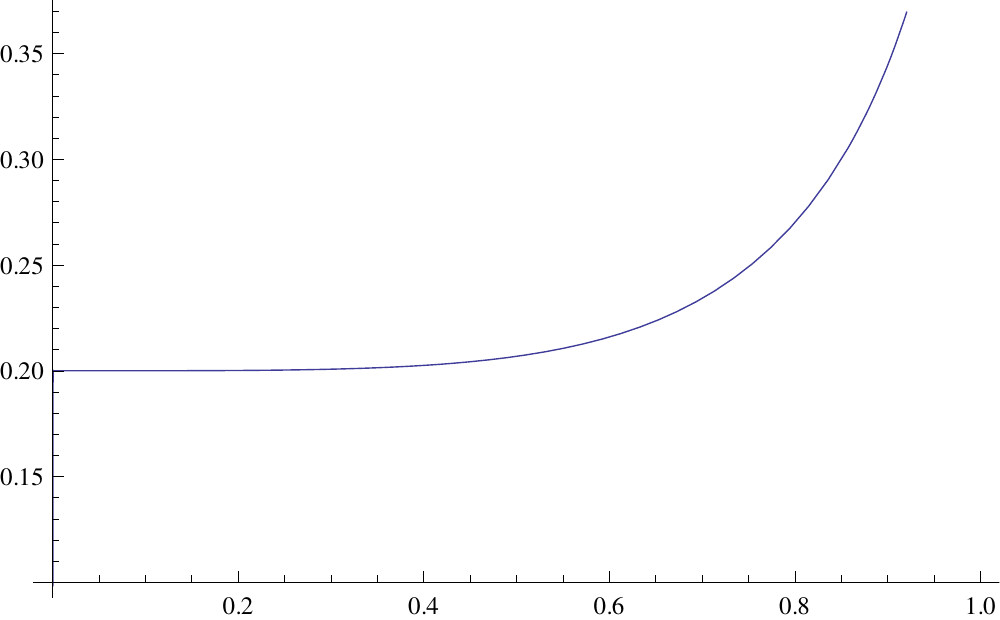}
\caption{Graph of the positive function $[1-16\sqrt{1-\kappa^2+\kappa^4}K^2(\kappa)F'(\kappa)G(\kappa)]$}
\label{fig7}
\end{figure}
\end{proof}

\subsubsection{The case $p=3$}
Consider
  \begin{equation}\label{mkdv3}
  \mathcal{L}=-\partial_x^2+w-3\varphi^2.
  \end{equation}
  We use (\eqref{mkdv2})  to
rewrite the operator $\mathcal{L}$ in an  appropriate form. From
the expression for $\varphi(x)$ from (\ref{mkdv2})
     and the relations
     between the elliptic functions $sn(x)$, $cn(x)$ and $dn(x)$, we obtain
       $$\mathcal{L}=\alpha^{2}[ -\partial_{y}^{2}+6k^{2} sn^{2}(y)-4-k^2] $$
     where $y=\alpha x$.

     It is well-known \cite{HIK}
     that the first five eigenvalues of
     $\Lambda =-\partial_{y}^{2}+6k^{2}sn^{2}(y, k)$,
     with periodic boundary conditions on $[0, 4K(k)]$, where
     $K(k)$ is the complete elliptic integral of the first kind, are
     simple. These eigenvalues, with their  corresponding eigenfunctions are as follows
      $$\begin{array}{ll}
         \nu_{0}=2+2k^2-2\sqrt{1-k^2+k^4},
         & \psi_{0}(y)=1-(1+k^2-\sqrt{1-k^{2}
         +k^{4}})sn^{2}(y, k),\\[1mm]
         \nu_{1}=1+k^{2}, & \psi_{1}(y)=cn(y, k)dn(y, k)
         =sn'(y, k),\\[1mm]
         \nu_{2}=1+4k^{2}, & \psi_{2}(y)=sn(y, k)dn(y, k)
         =-cn'(y, k),\\[1mm]
         \nu_{3}=4+k^{2}, & \psi_{3}(y)=sn(y, k)cn(y, k)
         =-k^{-2}dn'(y, k),\\[1mm]
         \nu_{4}=2+2k^{2}+2\sqrt{1-k^{2}+k^{4}},
         & \psi_{4}(y)=1-(1+k^{2}+\sqrt{1-k^{2}
         +k^{4}})sn^{2}(y, k).
        \end{array}
      $$
      It follows that the first three eigenvalues of the operator
      $\mathcal{L}$, equipped with periodic boundary condition on $[0,2K(k)]$
      (that is, in the case of left and right family),
      are simple and $\lambda_0=\alpha^2(\nu_0-\nu_3)<0, \;
      \lambda_1=\alpha^2(\nu_3-\nu_3)=0, \;
      \lambda_{2}=\alpha^2(\nu_4-\nu_3)>0$.
The corresponding eigenfunctions are $\chi_0=\psi_0(\alpha x),
\chi_1=\varphi'(x), \chi_2=\psi_4(\alpha x)$.
Thus, we have proved the following
\begin{proposition}
\label{prop:3}
 The linear
operator $\cL$ defined by $\eqref{mkdv3}$ has the following
spectral properties:
\begin{itemize}
\item[(i)]
  The first three eigenvalues of
$\cL$ are simple. \\
\item[(ii)] \it The second eigenvalue of $\cL$
is $\la_1=0$, which is simple. \\
\item[(iii)] The rest of the spectrum   consists of a  discrete
set of eigenvalues, which are strictly positive.
\end{itemize}
\end{proposition}
Next, we verify the following technical result.
\begin{lemma}
\label{l2}
The operator $\cL$ verifies $ \langle \mathcal{L}^{-1}1,1\rangle \neq 0.$
\end{lemma}

\begin{proof}
This statement was needed and proved in the work of Deconinck-Kapitula, \cite{DK}, but we repeat the short argument for completeness. 
First we will prove that for $i\neq 0, 4$, $\langle \psi_i, 1\rangle=0.$

  Using the expression for $\Lambda$ and $\psi_4$, we get
    $$\nu_i \langle \psi_i,1\rangle =6\kappa^2 \langle sn^2(y; \kappa), \psi_i\rangle={\frac{12\kappa^2}{\nu_4}}\langle 1-\psi_4, \psi_i\rangle.
    $$
It follows that for $i\neq 4$,
\begin{equation}
\label{300}
0=\langle \psi_4, \psi_i\rangle=\left( 1-{\frac{\nu_i
\nu_4}{12\kappa^2}}\right)\langle \psi_i,1\rangle .
\end{equation}
Observe however that $\nu_0 \nu_4=12\kappa^2$, which means that $  \left( 1-{\frac{\nu_i
\nu_4}{12\kappa^2}}\right)\neq 0$, whenever $i\neq 0$ (since $\nu_i\neq \nu_0$). By \eqref{300}, this implies that $\dpr{\psi_i}{1}=0, i\neq 0,4$.

From (Theorem 2.15, in \cite{MaWi}), the eigenfunctions of the
operator $\cL$ form an orthonormal  basis of  $L^2[0,T]$ and hence
we compute $\dpr{\cL^{-1} 1}{1}$  by expanding $1$ in the
eigenfunction expansion. Note that all  terms corresponding to
mean value zero eigenfunctions disappear (since $\dpr{\psi_i}{1}=0, i\neq 0,4$).
Hence, the expansion for $\dpr{\cL^{-1} 1}{1}$ has only two
non-zero terms.  More precisely, we have
  \begin{eqnarray}
\nonumber
  \langle \mathcal{L}^{-1}1,1\rangle & = & {\frac{ \langle 1,
  \chi_0\rangle ^2}{\alpha^2(\nu_0-\nu_3)||\chi_0||^2}}+{\frac{ \langle 1,
  \chi_2\rangle ^2}{\alpha^2(\nu_4-\nu_3)||\chi_2||^2}}\\
 \label{mkdv3a}
  &=&
  {\frac{2}{\alpha^3}}\left[
{\frac{B_1(\kappa)}{(\kappa^2-2-2\sqrt{1-\kappa^2+\kappa^4})B_3(\kappa)}}+
  {\frac{B_2(\kappa)}{(\kappa^2-2+2\sqrt{1-\kappa^2+\kappa^4})B_4(\kappa)}}\right]
  \end{eqnarray}
where we have used the formulas
\begin{eqnarray*}
 & &  \int_{0}^{2K(\kappa)} sn^2(y) dy=\frac{2}{\kappa^2} [K(\kappa)-E(\kappa)]; \\
 & &  \int_{0}^{2K(\kappa)}{sn^4(y)}dy={\frac{2}{3\kappa^4}}[(2+\kappa^2)K(\kappa)-2(1+\kappa^2)E(\kappa)]
\end{eqnarray*}
 and
 \begin{eqnarray*}
  B_1(\kappa) &=&  \left(
  {\frac{\sqrt{1-\kappa^2+\kappa^4}-1}{\kappa^2}}K(\kappa)+{\frac{1+\kappa^2-\sqrt{1-\kappa^2+\kappa^4}}{\kappa^2}}E(\kappa)\right)^2, \\
  B_2(\kappa) &=& \left(-{\frac{\sqrt{1-\kappa^2+\kappa^4}+1}{\kappa^2}}K(\kappa)+{\frac{1+\kappa^2+\sqrt{1-\kappa^2+\kappa^4}}{\kappa^2}}E(\kappa)\right)^2,\\
  B_3(\kappa)&= &K(\kappa)-{\frac{2(1+\kappa^2-\sqrt{1-\kappa^2+\kappa^4})}{\kappa^2}}[K(\kappa)-E(\kappa)],\\
  \\
  &+&{\frac{(1+\kappa^2-\sqrt{1-\kappa^2+\kappa^4})^2}{3\kappa^4}}
  [(2+\kappa^2)K(\kappa)-2(1+\kappa^2)E(\kappa)]; \\
  B_4(\kappa)&= & K(\kappa)-{\frac{2(1+\kappa^2+\sqrt{1-\kappa^2+\kappa^4})}{\kappa^2}}[K(\kappa)-E(\kappa)]\\
  \\
  &+& {\frac{(1+\kappa^2+\sqrt{1-\kappa^2+\kappa^4})^2}{3\kappa^4}}
  [(2+\kappa^2)K(\kappa)-2(1+\kappa^2)E(\kappa)].
  \end{eqnarray*}
 From the graph of the function ${\frac{B_1(\kappa)}{(\kappa^2-2-2\sqrt{1-\kappa^2+\kappa^4})B_3(\kappa)}}+
  {\frac{B_2(\kappa)}{(\kappa^2-2+2\sqrt{1-\kappa^2+\kappa^4})B_4(\kappa)}}$ below, we realize that $\dpr{\cL^{-1} 1}{1}>0$ and hence Lemma \ref{l2} is established.
    \begin{figure}[h8]
\centering
\includegraphics[width=8cm,height=6cm]{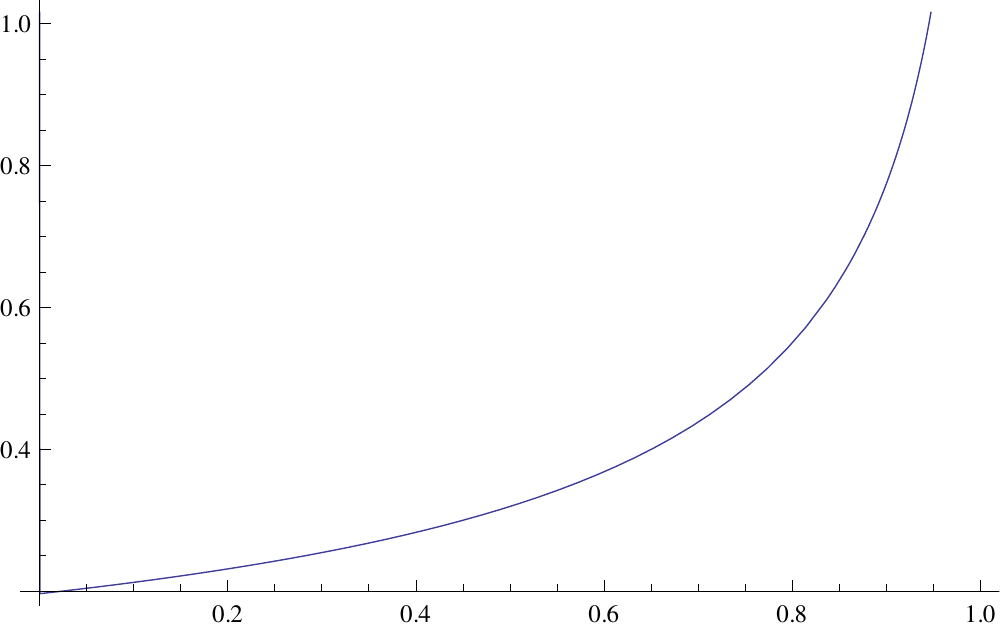}
\caption{Graph of the positive function ${\frac{B_1(\kappa)}{(\kappa^2-2-2\sqrt{1-\kappa^2+\kappa^4})B_3(\kappa)}}+
  {\frac{B_2(\kappa)}{(\kappa^2-2+2\sqrt{1-\kappa^2+\kappa^4})B_4(\kappa)}}$}
\label{fig8}
\end{figure}
\end{proof}

\pagebreak
Lemma \ref{l1} and Lemma \ref{l2} allow us to verify the important property about the simplicity of the zero eigenvalue for the operator $\ch_c$.
\begin{corollary}
\label{cor:p}
Let $p=2$ or $p=3$. Then, the operator $\ch_c=-\p_x \cL\p_x $ (corresponding to $f(z)=z^3, \f{z^2}{2}$) defined in \eqref{z:2}
has zero as a simple eigenvalue in $L^2_0[0,T]$, with an eigenfunction \\  $\psi_0=\vp_c-\f{1}{T}\int_0^T \vp_c$.
\end{corollary}
\begin{proof}
First $\vp_c-\f{1}{T}\int \vp_c(x) dx$ is easily seen to be an eigenfunction, since
$$
\ch_c[\vp_c-\f{1}{T}\int \vp_c(x) dx]=-\p_x \cL[\vp_c']=0,
$$ since $\vp_c'$ is an eigenfunction for $\cL$.

Regarding uniqueness, let $f\in L^2_0[0,T]$, so that $\ch_c f=0$. It follows that
$$
\cL[f']=c=const.
$$
Since $Ker(\cL)=span \{\vp_c'\}$ and $1\perp \vp_c'$, we can resolve the last equation as $f'=c \cL^{-1} 1$. Thus
$$
0=\dpr{1}{f'}=c\dpr{1}{\cL^{-1} 1},
$$
whence $c=0$, since $\dpr{1}{\cL^{-1} 1}\neq 0$ by Lemma \ref{l2}. It follows that $f'$ is an eigenvector for $\cL$. Thus,
$f'=\mu \vp_c'$, by Proposition \ref{prop:3}. But then,
$f=\mu[\vp_c-\f{1}{T}\int \vp_c(x) dx]$, since we are in the space $L^2_0(0,T)$ and the uniqueness is established.
\end{proof}
\subsection{Spectral theory for the Schr\"odinger operator $\ch$ of the KGZ system}
\label{sec:3.3}
First, as in the Boussinesq case, we show that the operator $\ch$ has a simple eigenvalue at zero. In addition, we identify the unique
(up to a multiplicative constant) eigenfunction of $\ch$. Recall that in our considerations, we work with the space $L^2_0[0,T]$,
that is the second component will contain only functions with mean value zero. This proposition mirrors closely the corresponding statement
of Proposition 8 in \cite{SS1}, with a few notable differences.
\begin{proposition}
\label{prop:k1}
The self-adjoint operator $\ch$ introduced in \eqref{k:22} has an eigenvalue at zero, which is simple. In addition, the unique (up to a multiplicative constant)
eigenfunction    is given by
$\vec{\psi}_0=\left(\begin{array}{c} \vp_c' \\
-\f{1}{2w}\left(\vp_c^2-T^{-1}\int_0^T \vp_c^2\right)
\end{array}\right)$.
\end{proposition}
\begin{proof}
Let $\left(\begin{array}{c} f \\ g
\end{array}\right)$ be an eigenvector corresponding to a zero eigenvalue, that is $\ch \left(\begin{array}{c} f \\ g
\end{array}\right)=0$. In other words,
\begin{equation}
\label{k:50}
\left\{
\begin{array}{l}
-w f''+f-\f{\vp^2}{2w} f+\vp g'=0 \\
-(\vp f)'-w g''=0.
\end{array}\right.
\end{equation}
Integrating the second equation in $x$ implies that for some constant $c_0$, we have
$$
g'=-\f{\vp f}{w}+c_0,
$$
whence the equation for $f$ becomes
\begin{equation}
\label{k:60}
 -w f''+f-\f{3\vp^2}{2w} f+c_0 \vp=0
\end{equation}
We will show that $c_0=0$ and then $f=d \vp_c'$ for some constant $d$.
To that end, recall the defining equation for $\vp_c$, namely \eqref{k:5} and differentiate it with respect to $x$. We get
\begin{equation}
\label{k:70}
-w \vp''_c+\vp'_c-\f{3\vp^2}{2w} \vp'_c=0.
\end{equation}
Following the usual analogy with the KdV equation, we introduce the second order differential operator
$$
\cL=-w\p^2_{x}+1-\f{3\vp^2}{2w}.
$$

Using that $dn^2+\kappa^2sn^2=1$ and $w\alpha^2={\frac{1}{2-\kappa^2}}$, we get
  $$\begin{array}{ll}
    \cL &=-w\p^2_{x}+1-{\frac{3\vp_1^2}{2w}}(1-\kappa^2 sn^2(\alpha x, \kappa))\\
    \\
    &=-w\p^2_{x}+1-6w\alpha^2+6w\alpha^2\kappa^2sn^2(\alpha x, \kappa)\\
    \\
    &=w\alpha^2(-\p^2_{y}+6sn^2(y, \kappa)-\kappa^2-4),
    \end{array}
  $$
   where $y=\alpha x$. It follows that the first three eigenvalues of the operator
      $\mathcal{L}$, equipped with periodic boundary condition on $[0,2K(k)]$,
      are simple and $\lambda_0=w\alpha^2(\nu_0-\nu_3)<0, \;
      \lambda_1=w\alpha^2(\nu_3-\nu_3)=0, \;
      \lambda_{2}=w\alpha^2(\nu_4-\nu_3)>0$.
The corresponding eigenfunctions are $\chi_0=\psi_0(\alpha x),
\chi_1=\varphi'(x), \chi_2=\psi_4(\alpha x)$, where $\nu_i$ and $\psi_i$ are given in (3.2.2).

 In particular, the kernel of $\cL$ is spanned by $\vp_c'$, i.e. $Ker(\cL)=span[\vp_c']$. Going back to \eqref{k:60}, we can rewrite it as
$$
\cL[f]+c_0 \vp_c=0.
$$
Note that all solutions to  this last equation are given by
$$
f=d\vp' - c_0 \cL^{-1}[\vp_c],
$$
where $d$ is an arbitrary scalar, since $\vp_c\perp span[\vp_c']= Ker(\cL)$.
Plugging this last formula in the equation for $g$ yields
$$
g'=-\f{\vp f}{w}+c_0=-\f{\vp}{w}\left(d\vp'-c_0\cL^{-1}[\vp]\right)+c_0=
-\f{d}{w}\vp \vp'+c_0\left(\f{\vp \cL^{-1}[\vp]}{w}+1\right)
$$
Integrating the last expression in $[0,T]$ and using the periodicity  yields
\begin{equation}
\label{k:80}
c_0\left(\f{\dpr{\vp}{\cL^{-1}[\vp]}}{w}+T\right)=0.
\end{equation}
Thus, if we verify that $\dpr{\vp}{\cL^{-1}[\vp]}\neq - T w$, we would conclude from \eqref{k:80} that $c_0=0$, whence $f=d \vp'$. Furthermore,
$g'= -\f{d}{w}\vp \vp'$, whence
$$
g=-d\f{\vp^2}{2 w}+const.
$$
Recall however that the $const.$ in the formula above is uniquely determined by the fact that $g$ has mean value zero (i.e. $g\in L^2_0[0,T]$),  whence
$$
g=d\left(-\f{\vp^2}{2w }+ \f{\int_0^T \vp^2}{2 T w}\right).
$$
Thus, Proposition \ref{prop:k1} is established modulo the following\\
\\
{\bf Fact:} $\dpr{\vp_c}{\cL^{-1}[\vp_c]}>-\f{Tw}{3}$,  so in particular $\dpr{\vp}{\cL^{-1}[\vp]} > - T w$ \\
\\
  From (\ref{k:5}), we have
    $\vp_c^3=-2w^2\vp_c''+2w\vp_c$ and thus
    \begin{equation}
    \label{n:500}
    \cL\vp=-w\vp_c''+\vp_c-\f{3}{2w}\vp_c^3=2w\vp_c''-2\vp_c.
    \end{equation}
    On the other hand, differentiating $-w^2\vp_c''+w\vp_c-{\frac{\vp_c^3}{2}}=0$, with respect to $w$ (and dividing by $w$) and using \eqref{n:500} to express
    $\vp_c''$,   results in
    $$
    \cL\left[\frac{d\vp_c}{dw}\right]=2\vp''-{\frac{1}{w}}\vp_c={\frac{1}{w}}\vp_c+{\frac{1}{w}}\cL\vp_c,
    $$
Taking $\cL^{-1}$ in the last identity yields
    $$
    \cL^{-1}\vp_c=w{\frac{d\vp_c}{dw}}-\vp_c.
    $$
Since   $\int_{0}^{K(\kappa)}{dn^2(y,\kappa)}dy=E(\kappa)$, we compute $\int_{0}^{T}{\vp^2}dx=\frac{16w^2}{T} E(\kappa)K(\kappa)$.
    Therefore,
     \begin{eqnarray*}
         \langle \cL^{-1}\vp_c, \vp_c\rangle &=& w\langle \vp_c, {\frac{d\vp_c}{dw}}\rangle-\langle \vp_c, \vp_c\rangle=\f{w}{2} \p_w[\|\vp_c\|^2]- \|\vp_c\|^2= \\
         &=&
        \f{w}{2} \p_w[ \frac{16w^2}{T} E(\kappa)K(\kappa)]-
        \frac{16w^2}{T} E(\kappa)K(\kappa)=\frac{8w^3}{T}}{\frac{d}{d\kappa}}[K(\kappa)E(\kappa)]{\frac{d\kappa}{dw}
          \end{eqnarray*}

To compute $\frac{d\kappa}{dw}$, note that from (\ref{k:10a}), we have  $2w\alpha=\vp_1$, whence
              \begin{equation}\label{k:19a}
                4w(2-\kappa^2)K^2(\kappa)=T^2.
              \end{equation}
              Differentiating (\ref{k:19a}) respect to $w$, we
              get
              \begin{equation}\label{k:19b}
                {\frac{d\kappa}{dw}}=-{\frac{(2-\kappa^2)K^2(\kappa)}{w{\frac{d}{d\kappa}}[(2-\kappa^2)K^2(\kappa)]}}.
              \end{equation}
Thus, using \eqref{k:19a},
\begin{eqnarray*}
 \langle \cL^{-1}\vp_c, \vp_c\rangle &=&  -\frac{8w^2}{T} \frac{d}{d\kappa}[K(\kappa)E(\kappa)] {\frac{(2-\kappa^2)K^2(\kappa)}{{\frac{d}{d\kappa}}[(2-\kappa^2)K^2(\kappa)]}}= \\
 &=& -w T\left[2 \f{ \frac{d}{d\kappa}[K(\kappa)E(\kappa)] }{{\frac{d}{d\kappa}}[(2-\kappa^2)K^2(\kappa)]} \right]
\end{eqnarray*}
Looking at the graph below, we realize that since
$\f{1}{3}=\lim_{\ka\to 0} 2 \f{ \frac{d}{d\kappa}[K(\kappa)E(\kappa)] }{{\frac{d}{d\kappa}}[(2-\kappa^2)K^2(\kappa)]}   \geq 2 \f{ \frac{d}{d\kappa}[K(\kappa)E(\kappa)] }{{\frac{d}{d\kappa}}[(2-\kappa^2)K^2(\kappa)]}$,  we have that
$$
\langle \cL^{-1}\vp_c, \vp_c\rangle\geq -\f{1}{3} wT,
$$
which establishes the claim.
\begin{figure}[h9]
\centering
\includegraphics[width=8cm,height=6cm]{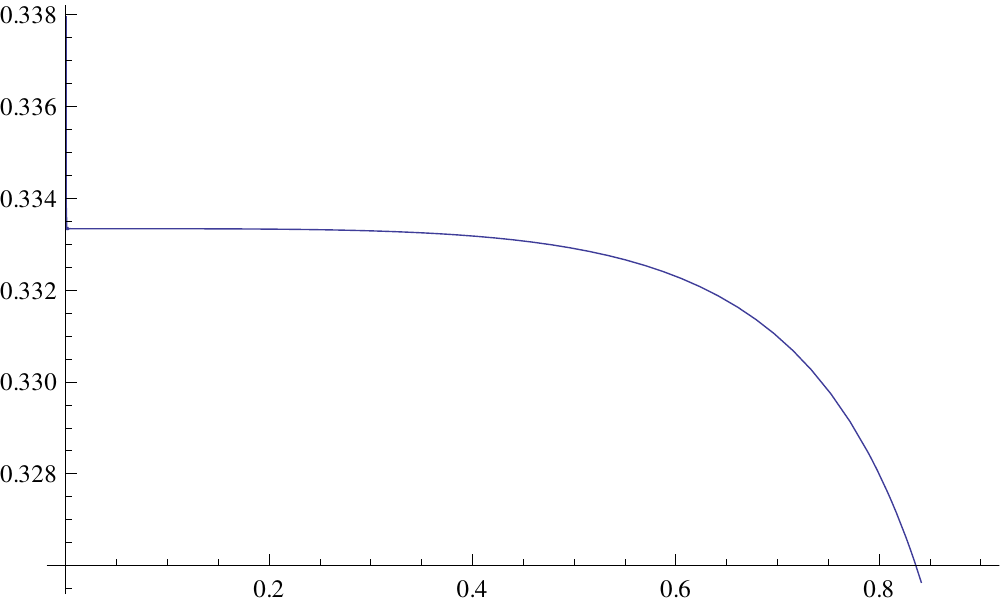}
\caption{Graph of the function
$2 \f{ \frac{d}{d\kappa}[K(\kappa)E(\kappa)] }{{\frac{d}{d\kappa}}[(2-\kappa^2)K^2(\kappa)]} $.}
\label{fig9}
\end{figure}

\end{proof}
The next thing one needs to establish, in order to apply Theorem \ref{theo:5} is that the operator $\ch$ for the KGZ system ( defined in \eqref{k:22} )
has a simple negative eigenvalue. This result should be compared with the corresponding statement in Proposition 9 in  \cite{SS1} for the whole line case.
\begin{proposition}
\label{prop:k2}
The operator $\ch$, defined in \eqref{k:22} has a simple negative eigenvalue.

\end{proposition}
\begin{proof}
Consider the eigenvalue problem in the form
$$
\ch\left(\begin{array}{c} f \\ g \end{array}\right)=-a^2 \left(\begin{array}{c} f \\ g \end{array}\right).
$$
for some $a\in (0,\infty)$. As in Proposition \ref{prop:k1}, this can be rewritten as\footnote{Note that the second equation requires $\int_0^T g(x) dx=0$.
This means in particular  that the statement for existence of negative eigenvalue  is invalid, unless the second component of the Hilbert space is $L^2_0[0,T]$}
\begin{equation}
\label{k:100}
\left\{
\begin{array}{l}
-w f''+f-\f{\vp^2}{2w} f+\vp g'=-a^2 f \\
-(\vp f)'-w g''=-a^2 g.
\end{array}\right.
\end{equation}
Form the second equation, we may resolve for $g$
\begin{equation}
\label{k:110}
g=  (a^2-w \p_x^2)^{-1}\p_x [\vp_c f].
\end{equation}
This last formula requires a bit of justification, but basically since $\p_x[\vp f]$ is guranteed to have mean value zero, it suffices to define
$(a^2-w \p_x^2)^{-1}$ (where $a^2>0, w>0$) on $L^2[0,T]$ by
$$
(a^2-w \p_x^2)^{-1}[\sum_{n=-\infty}^\infty  a_n e^{i n x}]:=
\sum_{n=-\infty}^\infty
\f{a_n}{a^2+4\pi^2 w \f{n^2}{T^2}} e^{2\pi i n \f{x}{T}},
$$
whence the formula for $g$ in \eqref{k:110} makes sense. In fact,  $L^2_0[0,T]$ is invariant under the action of  and $(a^2-w \p_x^2)^{-1}$ and
hence $g\in L^2[0,T]$. In fact, we use \eqref{k:110} to deduce the following formula for $g'$
$$
g'=\p_x^2 (a^2-w \p_x^2)^{-1} [\vp_c f]= -\f{\vp f}{w} +\f{a^2}{w} (a^2-w \p_x^2)^{-1} [\vp_c f].
$$
We plug this in the first equation of \eqref{k:100} to obtain the following equation for $f$
\begin{equation}
\label{k:120}
-w f''+(1+a^2) f-\f{3\vp_c^2}{w} f+\f{a^2}{w}[\vp_c (a^2-w \p_x^2)^{-1} \vp_c f] =0.
\end{equation}
Introduce a one-parameter family of self-adjoint operators
$$
M_a:=-w\p_x^2+(1+a^2) - \f{3\vp^2}{2w}+\f{a^2}{w}[\vp_c (a^2-w \p_x^2)^{-1} (\vp_c\cdot)].
$$
In order to finish the proof of Proposition \ref{prop:k2}, we need to establish that there is unique $a_0>0$, so that the operator $M_{a_0}$
has an eigenvalue zero and such an eigenvalue is simple. We shall first show that there exists $a_0>0$, so that $M_{a_0}$ has an eigenvalue at
zero and then we show that this eigenvalue zero is simple for $M_{a_0}$.

To that end,  we will establish the following \\
{\bf Claim:} For $a\geq b\geq 0$, we have $M_a\geq M_b+(a^2-b^2)Id\geq M_b$.

Assuming the validity of the Claim, let us finish the proof of Proposition
\ref{prop:k2}.   Define $\la(a)$ to be the minimal eigenvalue for $M_a$, that is
$$
\la(a):=\inf\{\la: \la \in \si(M_a)\}=\inf_{\|f\|=1} \dpr{M_a f}{f}.
$$
Clearly, the function $a\to \la(a)$ is continuous (in fact, more generally,  $a\to M_a$ is continuous as a function from $\rone \to B(L^2[0,T])$)).
Moreover, as a consequence of the Claim, $\la(a)$ is a strictly increasing function of its argument. In order to show the existence of $a_0$,
it will suffice to show that $a\to \la(a)$ changes sign in
$[0, \infty)$ (and hence vanishes at some $a_0>0$).  But at $a=0$, we have
$$
M_0=-w \p_x^2+1-\f{3\vp^2}{2 w}=\cL,
$$
that was considered before. Since we have checked that $\cL$ has a (simple) negative eigenvalue, say $-\de^2$ it follows that $\la(0)<0$.
On the other hand, by the claim
$$
M_a\geq M_0+a^2 Id=\cL+a^2 Id.
$$
In particular, for every $a>\de$, we have $M_a>(a^2-\de^2) Id$, whence $\la(a)\geq a^2-\de^2>0$. Thus, for any $b>\de$, the function $\la(a)$
changes sign (exactly once) in the interval $(0,b)$. We have shown that there is $a_0: \la(a_0)=0$.

Now, we have to show the second part of the proposition, namely that $0$ is an isolated eigenvalue for $M_{a_0}$. Let $\phi_0$ be the
eigenvector for the simple negative eigenvalue for $M_0=\cL$. Note that by  Proposition \ref{prop:1}, the second eigenvalue of $\cL$ is zero,
which means that $M_0|_{{\phi_0}^\perp}=\cL|_{{\phi_0}^\perp}\geq 0$.
 In addition, by the Claim, we have that $M_{a_0}\geq M_0+a_0^2 Id=\cL+a_0^2 Id$, and thus, by the Courant maxmin principle for the second eigenvalue
 $$
 \la_1(M_{a_0})=\sup_{z\neq 0} \inf_{u\perp z: \|u\|=1} \dpr{M_{a_0} u}{u}\geq
 a_0^2 + \inf_{u\perp \phi_0: \|u\|=1} \dpr{\cL u}{u}\geq a_0^2>0.
 $$
 It follows that $\la_0(M_{a_0})=0$, while $\la_1(M_{a_0})>0$, which shows the simplicity of the zero eigenvalue for $M_{a_0}$. It now remains to establish the Claim.
 \begin{proof}
 By the form of the operators $M_a$, it suffices to show for all trig polynomials  $f\in L^2[0,T]$ that for $a\geq b$,
 $$
 a^2\dpr{\vp_c (a^2-w \p_x^2)^{-1} [\vp_c f]}{f}\geq b^2\dpr{\vp_c (b^2-w \p_x^2)^{-1} [\vp_c f]}{f}
 $$
 But letting $a_n$ be the Fourier coefficients of $\vp_cf$, that is
 $\vp_c f=\sum_n a_n \f{e^{2\pi i n \f{x}{T}}}{\sqrt{T}}$  or \\
 $a_n=\f{1}{\sqrt{T}} \int_0^T \vp_c(x) f(x)  e^{-2\pi i n \f{x}{T}} dx$,
 we see that
 \begin{eqnarray*}
  b^2\dpr{\vp_c (b^2-w \p_x^2)^{-1} [\vp_c f]}{f} &=&
 \sum_n \f{b^2  }{b^2+4\pi^2 w \f{n^2}{T^2}}|a_n|^2\leq \\
 & \leq &   \sum_n \f{a^2  }{a^2+4\pi^2 w \f{n^2}{T^2}}|a_n|^2= a^2\dpr{\vp_c (a^2-w \p_x^2)^{-1} [\vp_c f]}{f},
 \end{eqnarray*}
where we have used that  $\f{b^2  }{b^2+4\pi^2 w \f{n^2}{T^2}}\leq \f{a^2  }{a^2+4\pi^2 w \f{n^2}{T^2}}$, whenever $w>0$, $a\geq b\geq 0$.
 \end{proof}

 \end{proof}

\section{Linear stability for Boussinesq: proof of  Theorem \ref{theo:1} and Theorem \ref{theo:2}}
\label{sec:133}
Now that we have the solutions, we need to check that the operator $H_c$ satisfies the requirements in Theorem \ref{theo:5}, after which,
we need to compute the index $\mu^*(H_c)$.
We collect the needed results in the following propositions.
\begin{proposition}(Spectral properties of $H_c$)
\label{prop:1}
The operator $\ch_c$, introduced in \eqref{z:2},
satisfies \eqref{A}, \eqref{E}, \eqref{B} for $p=2,3$.
\end{proposition}
Our next result gives a precise formula for the index $\mu^*(\ch_c)$.
\begin{proposition}
\label{prop:2} We have

1)for $p=2$
$$\mu^*(\ch_c)={\frac{\sqrt{w}}{2}}\sqrt{\widetilde{F}(\kappa)},
$$
where $\widetilde{F}(\kappa)$ is defined in \eqref{i:1} below. 

2) for  $p=3.$
$$
\mu^*(\ch_c)=\f{ \sqrt{w}}{2}
\sqrt{\f{\left[4E(\ka)-\f{\pi^2}{K(\ka)}\right][(2-\ka^2)
E(\ka)-2(1-\ka^2)K(\ka)]}{(2-\ka^2)(E^2(\ka)-(1-\ka^2) K(\ka))}}.
$$
\end{proposition}
{\bf Note:} The function under the square root is positive for all values of $\ka: 0<\ka<1$.

We now finish the proof of Theorem \ref{theo:1}, based on the results of Propositions \ref{prop:1} and \ref{prop:2}.

Let $p=3$. We apply Theorem \ref{theo:5}, whence we get stability, provided $|c|\geq \mu^*(\ch)$. Thus, we need to resolve the inequality
$$
|c|\geq  \f{\sqrt{1-c^2}}{2}
\sqrt{M(\ka)},
$$
where we have taken
$$
M(\ka):=\f{\left[4E(\ka)-\f{\pi^2}{K(\ka)}\right][(2-\ka^2) E(\ka)-2(1-\ka^2)K(\ka)]}{(2-\ka^2)(E^2(\ka)-(1-\ka^2) K(\ka))},
$$
we obtain for the interval of the stable speeds
$$
|c|\geq \sqrt{\f{M(\ka)}{4+M(\ka)}},
$$
as stated in Theorem \ref{theo:1}.

Similarly, for $p=2$, we have according to Proposition \ref{prop:2} that
$$
|c|\geq  \f{\sqrt{1-c^2}}{2}
\sqrt{\tilde{F}(\ka)},
$$
whence we conclude similarly that $$
|c|\geq \sqrt{\f{\tilde{F}(\ka)}{4+\tilde{F}(\ka)}},
$$
is a necessary and sufficient condition for stability of the corresponding traveling wave.
\subsection{Proof of Proposition \ref{prop:1}}
We note that standard arguments imply the validity of \eqref{E}, since $\ch$ is fourth order operator. The reality condition \eqref{B} is also trivially satisfied as all our potentials are real-valued. The hard to check condition is \eqref{A}.
 We need to check that the operator $\ch$ has a simple eigenvalue at zero. This was indeed the conclusion of Corollary \ref{cor:p}.

 Thus, it remains to show that the operator $\ch$ defined in \eqref{z:2} has a simple negative eigenvalue and Proposition
 \ref{prop:1} will follow. This is a non-trivial fact. Interestingly enough, this   was needed (and proved in our paper \cite{HSS}), when we have considered  the transverse instability of the same spatially periodic waves in the Kadomtsev-Petviashvili (KP) and modified  KP models, that is exactly for the operators considered here, corresponding to the cases $f(u)=\f{u^2}{2}$ and $f(u)=u^3$.  More precisely, we have derived a necessary condition, so that the first two  eigenvalues of $\ch=-\p_x\cL\p_x$ satisfy
 $$
\la_0( \ch)< \la_1(\ch) =0.
 $$
 which was verified for both $p=2,3$. This is exactly what is needed here. The interested reader may consult Section 4.1 in \cite{HSS} for full and complete proof of these facts.

\subsection{Proof of Proposition \ref{prop:2}}

\subsubsection{The case $p=2$}
We compute the index of stability $\mu^*(\ch)$. We have
  \begin{equation}\label{kdv9}
    \langle \ch^{-1}\psi_0', \psi_0'\rangle ={\frac{1}{\|\varphi_c+A\|^2}} \langle \ch^{-1}\varphi'_c, \varphi'_c\rangle
    ,
    \end{equation}
    where $A=-{\frac{1}{T}}\int_{0}^{T}{\varphi_c}dx.$ Let $f:
    \ch[f]=\varphi'_c$. It follows that $-\cL f'=\varphi_c+b$, for
    some constant $b$. Hence
      \begin{equation}\label{kdv9a}
        -f'=\cL^{-1}\varphi_c + b\cL^{-1}1.
      \end{equation}
      Thus
      \begin{equation}\label{kdv9b}
        \langle \ch^{-1}\psi_0',
        \psi_0'\rangle={\frac{1}{\|\varphi_c+A\|^2}}\langle f,
        \varphi'_c\rangle={\frac{1}{||\varphi_c+A||^2}}\langle \cL^{-1}\varphi_c,
        \varphi_c\rangle +b\langle \cL^{-1}1, \varphi_c\rangle .
      \end{equation}
   From (\ref{kdv9a}), we have that $0=-\dpr{f'}{1}=\dpr{\cL^{-1}\varphi_c}{1}+b\dpr{\cL^{-1}1}{1}$, whence
     \begin{equation}\label{kdv9c}
       b=-\frac{\langle \cL^{-1}\varphi_c, 1\rangle}{\langle
       \cL^{-1}1,1\rangle }
     \end{equation}
Combining (\ref{kdv5}), (\ref{kdv9}), and  (\ref{kdv9c}) yields
       \begin{eqnarray*}
        \langle \ch^{-1}\psi_0', \psi_0'\rangle
        &= & {\frac{1}{||\varphi_c+A||^2}}\left( \langle \cL^{-1}\varphi_c,
        \varphi_c \rangle -{\frac{\langle \cL^{-1}\varphi_c, 1 \rangle\langle \cL^{-1}1,
        \varphi_c\rangle}{\langle
       \cL^{-1}1,1\rangle }}\right) \\
        \\
        &=& {\frac{1}{||\varphi_c+A||^2}}\left( {\frac{1}{4c}}{\frac{d}{dc}}\langle \varphi_c, \varphi_c\rangle-{\frac{\langle \cL^{-1}\varphi_c, 1\rangle \langle \cL^{-1}1,
        \varphi_c\rangle}{\langle
       \cL^{-1}1,1\rangle }}\right).
       \end{eqnarray*}
  From (\ref{kdv1}) after integrating
  $$
  \int_0^T \vp_c^2 dx=2w\int_{0}^{T}{\vp_c}dx={\frac{2w}{T}}F(\kappa),
  $$
  we use \eqref{kdv7a} to see
  \begin{equation}
  \label{kdv10}
  \frac{1}{4c}\frac{d}{dc}\langle \varphi_c,
    \varphi_c\rangle = {\frac{1}{T}}\left[
    -F(\kappa)-256K^4(\kappa)F'(\kappa)G(\kappa)(1-\kappa^2+\kappa^4)\right].
    \end{equation}
Next,  we use \eqref{kdv5}, \eqref{kdv7b}, \eqref{kdv7a} and \eqref{n:10} to compute
    \begin{eqnarray*}
      \langle \cL^{-1}1, \varphi_c\rangle &=& \langle
      \cL^{-1}\varphi_c,1\rangle= \f{1}{2c}(\p_c\int_0^T \vp_c dx)= -wT^3F'(\kappa)G(\kappa)= \\
      &=& {-\frac{256}{wT}}[ K^4(\kappa)F'(\kappa)G(\kappa)(1-\kappa^2+\kappa^4)]
      \end{eqnarray*}
Finally,  using the formulas for $\int_0^T \vp_c^2 dx, \int_0^T \vp_c dx$ allows us to find
      $$
 ||\varphi_c+A||^2 = {\frac{w\left[
       2F(\kappa)-{\frac{F^2(\kappa)}{16\sqrt{1-\kappa^2+\kappa^4}K^2(\kappa)}}\right]}{T}}
    $$
 Putting all formulas together
     \begin{eqnarray*}
      & &  \langle \ch^{-1}\psi', \psi'\rangle={\frac{1}{w\left[
       2F(\kappa)-{\frac{F^2(\kappa)}{16\sqrt{1-\kappa^2+\kappa^4}K^2(\kappa)}}\right]}} \times    \\
       \\
     & & \times \left[-F(\kappa)-256K^4(\kappa)F'(\kappa)G(\kappa)(1-\kappa^2+\kappa^4)-\f{4096 K^6(\ka)(1-\ka^2+\ka^4)^{3/2}(F'(\ka) G(\ka))^2}{1-16\sqrt{1-\kappa^2+\kappa^4}K^2(\kappa)F'(\kappa)G(\kappa)}\right]
     \end{eqnarray*}
     Thus, if we assign the function
     \begin{equation}
     \label{i:1}
     \tilde{F}(\kappa):=\f{\left[
       2F(\kappa)-{\frac{F^2(\kappa)}{16\sqrt{1-\kappa^2+\kappa^4}K^2(\kappa)}}\right] }{ F(\kappa)+256K^4(\kappa)F'(\kappa)G(\kappa)(1-\kappa^2+\kappa^4)+\f{4096 K^6(\ka)(1-\ka^2+\ka^4)^{3/2}(F'(\ka) G(\ka))^2}{1-16\sqrt{1-\kappa^2+\kappa^4}K^2(\kappa)F'(\kappa)G(\kappa)}}.
\end{equation}
 we get
     $\langle \ch^{-1}\psi', \psi'\rangle=-\frac{1}{w\tilde{F}(\kappa)} $. Thus, the index formula holds as stated in Proposition \ref{prop:2}, namely
     $$
     \mu^*(\ch)=\f{w}{2}\sqrt{\tilde{F}(\ka)}.
     $$
\subsubsection{The case $p=3$}
In this section,  we compute the index of stability.  For that we need to consider first
$$
\dpr{\ch^{-1} \psi_0'}{\psi_0'}=\frac{1}{\|\varphi_c+A\|^2}\dpr{\ch^{-1} \vp_c'}{\vp_c'},
$$
where $A=-{\frac{1}{T}}\int_{0}^{T}{\varphi_c}dx=-{\frac{\alpha\sqrt{2}\pi}{2K(\kappa)}}$. Thus, we need to compute $\ch^{-1}[\vp_c']$.
Let $f: \ch[f]=\vp_c'$. It follows that
$
-\cL f'=\vp_c+b,
$
for some constant $b$.  We conclude that
$$
-f'=\cL^{-1} \vp_c+b \cL^{-1} 1.
$$
Note that $\cL^{-1} 1$ is well-defined, since $1\perp \vp_c'$, which spans $Ker(\cL)$.
Thus
\begin{equation}
\label{350}
\dpr{\ch^{-1} \psi_0'}{\psi_0'}=\frac{1}{\|\varphi_c+A\|^2} \dpr{f}{\vp_c'}=
\frac{1}{\|\varphi_c+A\|^2} \dpr{-f'}{\vp_c}=
\frac{\dpr{\cL^{-1}\vp_c}{\vp_c}+b\dpr{\cL^{-1} 1}{\vp_c}}{\|\varphi_c+A\|^2}
\end{equation}
  Differentiating  \eqref{a:1010} (with $a=0$)  with respect to $c$ yields
    \begin{equation}\label{mkdv5}
        \mathcal{L}^{-1}\varphi_c={\frac{1}{2c}}{\frac{d\varphi_c}{dc}}
    \end{equation}
From \eqref{mkdv2}, we get $\int_{0}^{T}{\vp_c^2}dx={\frac{8K(\kappa)E(\kappa)}{T}}$ and
     \begin{equation}
     \label{mkdv6}
       \langle \mathcal{L}^{-1}\varphi_c, \varphi_c\rangle ={\frac{1}{2c}}\langle \frac{d\varphi_c}{dc},
       \varphi_c\rangle={\frac{1}{4c}}\p_c[\int_{0}^{T}{\varphi_c^2}dx] \\
       =-{\frac{4}{T}}{\frac{E^2(\kappa)-(1-\kappa^2)K^2(\kappa)}{\kappa
       (1-\kappa^2)}}{\frac{d\kappa}{dw}}.
     \end{equation}
     In addition, note that
     $$
             \dpr{\cL^{-1} 1}{\vp_c}= \langle \mathcal{L}^{-1}\varphi_c, 1\rangle
              ={\frac{1}{2c}}
              \p_c[\int_{0}^{T}{\varphi_c(x;\kappa)}dx]=0.
         $$
         and hence
         $$
         \dpr{\ch^{-1} \psi_0'}{\psi_0'}=\frac{\dpr{\cL^{-1}\vp_c}{\vp_c}}{\|\varphi_c+A\|^2}.
         $$
     From the relations \eqref{mkdv2}, we have
       $$
       w=\frac{2-\kappa^2}{2}\varphi_1^2=\frac{4 K^2(\ka)(2-\kappa^2)}{T^2},
       $$
       which after differentiating with respect to $w$ allows us to express
       $$
       {\frac{d\kappa}{dw}}={\frac{T^2}{8}}{\frac{1}{(2-\kappa^2)K(\kappa){\frac{dK(\kappa)}{d\kappa}}-\kappa
       K^2(\kappa)}}.
       $$
       Thus
         $$
         \dpr{\mathcal{L}^{-1}\varphi_c}{\varphi_c}
         =-\frac{T}{2}\frac{(E^2(\kappa)-(1-\kappa^2)K^2(\kappa))}{\kappa
       (1-\kappa^2)} \frac{1}{(2-\kappa^2)K(\kappa){\frac{dK(\kappa)}{d\kappa}}-\kappa
       K^2(\kappa)}.
       $$
       Using that
       ${\frac{dK(\kappa)}{d\kappa}}={\frac{E(\kappa)-(1-\kappa^2)K(\kappa)}{\kappa(1-\kappa^2)}}$,
       we obtain
         \begin{equation}\label{mkdv7}
         \langle \mathcal{L}^{-1}\varphi_c,
         \varphi_c\rangle=-{\frac{1}{\alpha}}{\frac{E^2(\kappa)-(1-\kappa^2)K^2(\kappa)}{(2-\kappa^2)E(\kappa)-2(1-\kappa^2)K(\kappa)}}=-{\frac{1}{\alpha}}B(\kappa).
         \end{equation}
       Since  $\int_{0}^{2K(\kappa)}{dn(y;\kappa)}dy=\pi$, we
       get
       \begin{equation}\label{mkdv8}
         ||\varphi_c+A||^2=\alpha \left(
         4E(\kappa)-{\frac{\pi^2}{K(\kappa)}}\right)=\alpha
         C(\kappa)
         \end{equation}

         From the above relations and \eqref{350}, we get
         \begin{equation}\label{mkdv10}
           \langle \ch^{-1}\psi_0',
           \psi_0'\rangle=-{\frac{2-\kappa^2}{wC(\kappa)}}
           B(\kappa)
         \end{equation}

Thus,
\begin{eqnarray*}
\mu^*(\ch) &=& \f{1}{2\sqrt{-\dpr{ \ch^{-1}\psi_0'}{\psi_0'}}}=
\f{\sqrt{w}}{2} \sqrt{\f{C(\ka)}{(2-\ka^2)B(\ka)}}= \\
&=& \f{\sqrt{w}}{2}\sqrt{\f{\left[4E(\ka)-\f{\pi^2}{K(\ka)}\right][(2-\ka^2) E(\ka)-2(1-\ka^2)K(\ka)]}{(2-\ka^2)(E^2(\ka)-(1-\ka^2) K(\ka))}},
\end{eqnarray*}
which  is exactly the claim of Proposition \ref{prop:2}.

\section{Linear stability of KGZ: proof of Theorem \ref{theo:10}}
\label{sec:5}
We have already checked the conditions on the operator $\ch$, defined in \eqref{k:22} in Section \ref{sec:3.3}. Namely, we have established the simplicity of the eigenvalue at zero in Proposition \ref{prop:k1} and then, we have verified the existence and simplicity of a single negative eigenvalue.  It now remains to compute the index $\mu^*(\ch)$,
after which, we obtain a characterization of the linear stability by Theorem \ref{theo:5}, namely $|c|\geq \mu^*(\ch)$.

\begin{proposition}
\label{prop:k15}
For $\ka \in (0,\ka_0), \ka_0=0.937095\ldots$, $\dpr{\ch^{-1}\psi_0'}{\psi_0'}<0$. For $\ka\in (\ka_0,1)$, \\ $\dpr{\ch^{-1}\psi_0'}{\psi_0'}>0$, and
$$
\mu^*(\ch)=\f{\sqrt{w}}{2 \sqrt{N(\ka)}},
$$
where $N$ is defined in \eqref{z001}. Note that,  $N(\ka)>0,\ka\in (\ka_0,1)$.
\end{proposition}
Assuming the validity of Proposition \ref{prop:k15}, we now finish the proof of Theorem \ref{theo:10}. To that end, observe that since  $\dpr{\ch^{-1}\psi_0'}{\psi_0'}<0, \ka\in (0,\ka_0)$, we have instability whenever $\ka\in (0,\ka_0)$. For $\ka \in (\ka_0, 1)$, we need to solve the inequality
$$
1>|c|\geq \f{\sqrt{w}}{2 \sqrt{N(\ka)}} = \f{\sqrt{1-c^2}}{2 \sqrt{N(\ka)}}
$$
which results in the following necessary and sufficient condition for linear stability
$$
1>|c|\geq \f{1}{\sqrt{1+4 N(\ka)}}, \ka \in (\ka_0, 1).
$$
This is exactly the statement of Theorem \ref{theo:10}.
\subsection{Proof of Proposition \ref{prop:k15}}

Now we will estimate the index of stability $\langle \ch^{-1} \psi_0', \psi_0'\rangle$,
$$
\psi_0=m\left( \begin{array}{cc}
  \vp_c'\\
  \\
  -{\frac{\vp_c^2}{2w}}+{\frac{1}{2wT}}\int_{0}^{T}{\vp_c^2}dx.
  \end{array}\right).
  $$
where $m$ is so that $\|\psi_0\|=1$.  Thus, we need to compute $ \ch^{-1}\left( \begin{array}{cc}
  \vp_c''\\
  -({\frac{\vp_c^2}{2w}})'.
  \end{array}\right)$ We have
    $$\left| \begin{array}{ll}
    -wf''+f-{\frac{\vp_c^2}{2w}}f+\vp_cg'=\vp_c'' \\
    \\
    -(\vp_cf)'-wg''=-\left( {\frac{\vp_c^2}{2w}}\right)'\end{array}
    \right. $$
    Integrating once in the second equation yields
      \begin{equation}\label{k:15a}
      g'={\frac{\vp_c^2}{2w^2}}+{\frac{c_1}{w}}-{\frac{\vp_cf}{w}},
      \end{equation}
      where $c_1$ is a constant of integration and needs to be determined.
The first equation becomes
        $$-wf''+f-{\frac{3\vp_c^2}{2w}}f+{\frac{\vp_c^3}{2w^2}}+{\frac{c_1\vp_c}{w}}=\vp_c''$$
        or
        \begin{equation}\label{k:16}
        \cL f+{\frac{\vp_c^3}{2w^2}}+{\frac{c_1\vp_c}{w}}=\vp_c''.
        \end{equation}
        On the other hand, taking derivative with
         respect to $w$ in (\ref{k:5}) yields
        \begin{equation}\label{k:17}
        \vp_c''=\cL{\frac{d\vp_c}{dw}}+{\frac{\vp_c^3}{2w^2}}.
        \end{equation}
        From (\ref{k:16}) and (\ref{k:17}), we have
          $$\cL\left(f-{\frac{d\vp_c}{dw}}\right)=-{\frac{c_1}{w}}\vp_c$$
          and hence
          \begin{equation}\label{k:18}
            f={\frac{d\vp_c}{dw}}-{\frac{c_1}{w}}\cL^{-1}\vp_c=(1-c_1){\frac{d\vp}{dw}}+{\frac{c_1}{w}}\vp_c.
          \end{equation}
Plugging this in (\ref{k:15a}) and integrating, we get
           \begin{equation}\label{k:19}
               c_1={\frac{\int_{0}^{T}{\vp_c{\frac{d\vp_c}{dw}}}dx-{\frac{1}{2w}}\int_{0}^{T}{\vp_c^2}dx}{T+{\frac{1}{w}}\langle\vp_c, \cL^{-1}\vp_c\rangle}}=
               {\frac{\int_{0}^{T}{\vp_c{\frac{d\vp_c}{dw}}}dx-{\frac{1}{2w}}\int_{0}^{T}{\vp_c^2}dx}
               {T+\int_{0}^{T}{\vp_c{\frac{d\vp_c}{dw}}}dx-{\frac{1}{w}}\int_{0}^{T}{\vp_c^2}dx}}.
            \end{equation}

              Using that $\int_{0}^{K(\kappa)}\vp_c^2=\f{16 w^2}{T} E(\kappa)K(\ka)$ and
              \eqref{k:19b}, we have
                \begin{eqnarray*}
                  \int_{0}^{T}{\vp_c{\frac{d\vp_c}{dw}}dx} &=& {\frac{1}{2}}{\frac{d}{dw}}\int_{0}^{T}{\vp_c^2}dx={\frac{1}{2}}{\frac{d}{dw}}[{\frac{16w^2}{T}}
                  E(\kappa)K(\kappa)] = \\
                  &=& {\frac{8w}{T}}\left[
                  2E(\kappa)K(\kappa)-{\frac{(2-\kappa^2)K^2(\kappa){\frac{d}{d\kappa}}[E(\kappa)K(\kappa)]}{{\frac{d}{d\kappa}}[(2-\kappa^2)K^2(\kappa)]}}\right].
                \end{eqnarray*}
The above formula allow us
 to  express  $c_1$   as a function of $\kappa$ only
 \begin{eqnarray*}
 c_1 &=& \frac{2E(\kappa)K(\kappa)\frac{d}{d\kappa}[(2-\kappa^2)K^2(\kappa)]-2(2-\kappa^2)K^2(\kappa)\frac{d}{d\kappa}[E(\kappa)K(\kappa)]}
                  {(2-\kappa^2)K^2(\kappa)\frac{d}{d\kappa}[(2-\kappa^2)K^2(\kappa)]-2(2-\kappa^2)K^2(\kappa)\frac{d}{d\kappa}[E(\kappa)K(\kappa)]}\\
                  \\
                  &=&{\frac{(2-\kappa^2)E^2(\kappa)-8(1-\kappa^2)E(\kappa)K(\kappa)+2(1-\kappa^2)(2-\kappa^2)K^2(\kappa)}
                  {2(2-\kappa^2)^2E(\kappa)K(\kappa)-2(1-\kappa^2)(2-\kappa^2)K^2(\kappa)-2(2-\kappa^2)E^2(\kappa)}}.
            \end{eqnarray*}

Now,
    $$\begin{array}{ll}
    \left\langle \ch^{-1}\psi_0', \psi_0'\right\rangle&=m^2\left\langle \ch \left(\begin{array}{cc} f\\g\end{array}\right),
     \left(\begin{array}{cc} f\\g\end{array}\right)\right\rangle=m^2\left\langle \left(\begin{array}{cc} \vp_c''\\-({\frac{\vp_c^2}{2w}})'\end{array}\right),
    \left(\begin{array}{cc} f\\g\end{array}\right)\right\rangle \\
    \\
    &=m^2\left(\langle\vp_c'',f\rangle+\langle g', {\frac{\vp_c^2}{2w}}\rangle \right).\end{array}$$
    From (\ref{k:5}) and the expression for $f$ and $g'$, we get
      $$ \langle \vp_c'', f\rangle ={\frac{1}{T}}(2J_2-J_3-J_4+J_5), \quad \langle g', {\frac{\vp_c^2}{2w}}\rangle={\frac{1}{T}}(J_1+J_2-J_3-J_4), $$
      where
\begin{eqnarray*}
I_1 &=& \int_{0}^{K(\kappa)}{dn^4(y, \kappa)}dy=\frac{4-2k^2}{3}E(k)-\frac{1-k^2}{3}K(k); \\
          J_1 &=& {\frac{T}{2w^3}}\langle \vp_c^2, \vp_c^2\rangle=16{\frac{K(\kappa)}{2-\kappa^2}}I_1;
          \\
          \\
          J_2 &=& {\frac{Tc_1}{2w^2}}\langle \vp_c, \vp_c\rangle=8c_1E(\kappa)K(\kappa);
          \\
          \\
          J_3 &=& {\frac{T(1-c_1)\langle \vp_c{\frac{d\vp_c}{dw}}, \vp_c^2\rangle}{2w^2}} =8(1-c_1)\left[ {\frac{3K(\kappa)}{2-\kappa^2}}I_1
          -{\frac{(2-\kappa^2)K^2(\kappa)}{{\frac{d}{d\kappa}}[(2-\kappa^2)K^2(\kappa)]}}{\frac{d}{d\kappa}}\left[ {\frac{K(\kappa)}{2-\kappa^2}}I_1\right]\right];\\
          \\
          J_4 &=& {\frac{Tc_1}{2w^3}}\langle \vp_c^2, \vp_c^2\rangle=32c_1{\frac{K(\kappa)}{2-\kappa^2}}I_1;\\
          \\
          J_5 &=&
          {\frac{T(1-c_1)}{w}}\langle \vp_c, {\frac{d\vp_c}{dw}}\rangle=8(1-c_1)
          \left[
          2K(\kappa)E(\kappa)-{\frac{(2-\kappa^2)K^2(\kappa){\frac{d}{d\kappa}}[E(\kappa)K(\kappa)]}{{\frac{d}{d\kappa}}[(2-\kappa^2)K^2(\kappa)]}}\right].
          \end{eqnarray*}
In addition, we have
          $${\frac{1}{m^2}}=\|\left( \begin{array}{cc}  \vp_c'\\
          -{\frac{\vp_c^2}{2w}}+{\frac{1}{2wT}}\int_{0}^{T}{\vp_c^2}dx
          \end{array}\right) \|^2=\langle \vp_c', \vp_c'\rangle+{\frac{1}{4w^2}}\langle \vp_c^2, \vp_c^2\rangle-{\frac{1}{4w^2T}}\left( \int_{0}^{T}{\vp_c^2}dx\right)^2$$
        Using that $dn'(y)=-\kappa^2sn(y)cn(y)$ and
        $sn^2(y)+cn^2(y)=1$, we get
          \begin{eqnarray*}
          \langle \vp_c', \vp_c'\rangle &=& \vp_1^2\alpha
          \kappa^4\left[ \int_{0}^{2K(\kappa)}{sn^2(y, \kappa)}dy-\int_{0}^{2K(\kappa)}{sn^4(y,
          \kappa)}dy\right]=\\
          \\
          &=& {\frac{8w}{3T}}{\frac{K(\kappa)}{2-\kappa^2}}[2(2-\kappa^2)E(\kappa)-4(1-\kappa^2)K(\kappa)].
          \end{eqnarray*}
          and
          \begin{eqnarray*}
            {\frac{1}{m^2}}&= & {\frac{w}{T}}\left[
            {\frac{8K(\kappa)}{3(2-\kappa^2)}}[2(2-\kappa^2)E(\kappa)-4(1-\kappa^2)K(\kappa)]+{\frac{16K(\kappa)}{2-\kappa^2}}I_1-
            {\frac{16}{2-\kappa^2}}E^2(\kappa)\right]=\\
            \\
            &=& {\frac{16 w}{T}}\left[E(\kappa)K(\kappa)-{\frac{1-\kappa^2}{2-\kappa^2}}K^2(\kappa)-{\frac{1}{2-\kappa^2}}E^2(\kappa)\right].
            \end{eqnarray*}
         Combining the above relations yields
\begin{equation}
\label{z001}
         \left\langle \ch^{-1}\psi_0', \psi_0'\right\rangle ={\frac{J_1+3J_2-2J_3-2J_4+J_5}{16 w\left[
            E(\kappa)K(\kappa)-{\frac{1-\kappa^2}{2-\kappa^2}}K^2(\kappa)-{\frac{1}{2-\kappa^2}}E^2(\kappa)\right]}}\\
             :={-\frac{N(\kappa)}{w   }}
\end{equation}
\begin{figure}[h10]
\centering
\includegraphics[width=8cm,height=6cm]{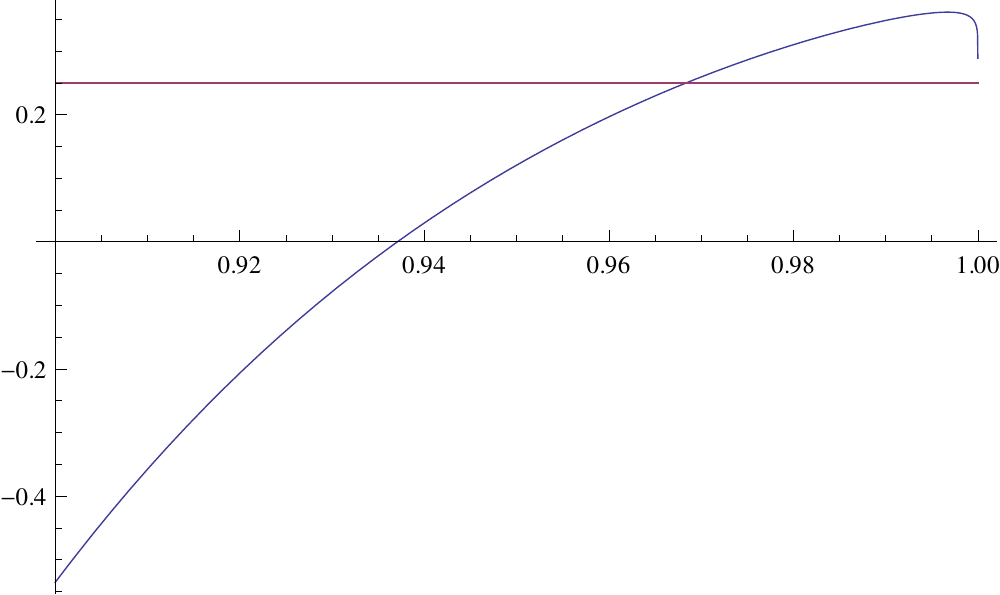}
\caption{The  function $N(\ka)$, together with $\f{1}{4}$. Recall that for stability, one needs $\left\langle \ch^{-1}\psi_0', \psi_0'\right\rangle<0$ and hence $N(\ka)>0$}
\label{fig10}
\end{figure}

\end{document}